\documentclass[12pt]{amsart}

\usepackage{amssymb, tikz}

\usepackage{amsmath}

\usepackage{hyperref}

\newtheorem{theorem}{Theorem}[section]
\newtheorem{lemma}[theorem]{Lemma}

\newtheorem{fact}[theorem]{Fact}
\newtheorem{question}[theorem]{Question}
\newtheorem{proposition}[theorem]{Proposition}
\newtheorem{corollary}[theorem]{Corollary}
\newtheorem{claim}[theorem]{Claim}
\theoremstyle{definition}
\newtheorem{definition}[theorem]{Definition}

\newtheorem{remark}[theorem]{Remark}

\let \restr = \upharpoonright

\let \into = \longrightarrow

\let \sub = \subseteq
\let \elsub = \preccurlyeq

\let \a = \alpha
\let \b = \beta
\let \g = \gamma
\let \d = \delta

\let \l = \lambda
\let \k = \kappa
\let \m = \mu
\let \n = \nu
\let \p = \pi
\let \r = \rho
\let \t = \theta

\let \s = \sigma
\let \x = \xi

\let \o = \omega

\let \al = \aleph

\let \mtcl = \mathcal
\let \mtbb = \mathbb
\let \it = \item

\DeclareMathOperator{\dom}{dom}
\DeclareMathOperator{\range}{range}
\DeclareMathOperator{\cf}{cf}
\DeclareMathOperator{\ot}{ot}
\DeclareMathOperator{\cl}{cl}

\DeclareMathOperator{\ssup}{sup}
\DeclareMathOperator{\FA}{FA}
\DeclareMathOperator{\Lim}{Lim}
\DeclareMathOperator{\Ord}{\textsf{Ord}}

\DeclareMathOperator{\MM}{\textsf{MM}}
\DeclareMathOperator{\CB}{CB}
\DeclareMathOperator{\wCC}{wCC}
\DeclareMathOperator{\ZFC}{ZFC}\DeclareMathOperator{\NS}{NS}
\DeclareMathOperator{\MA}{\textsf{MA}}
\DeclareMathOperator{\PFA}{\textsf{PFA}}
\DeclareMathOperator{\CH}{\textsf{CH}}
\DeclareMathOperator{\Unif}{\textsf{Unif}}

\title{Generalizations of Martin's Axiom, weak square, weak Chang's Conjecture, and a forcing axiom failure}

\author[D. Asper\'o]{David Asper\'o}

\address{David Asper\'o, School of Mathematics, University of East Anglia, Norwich NR4 7TJ, UK}

\email{d.aspero@uea.ac.uk}

\author[N. Tananimit]{Nutt Tananimit}

\address{Nutt Tananimit, School of Mathematics, University of East Anglia, Norwich NR4 7TJ, UK}

\email{n.tananimit@uea.ac.uk}

\date{}

\begin{document}



\maketitle
\pagestyle{myheadings}\markright{Generalizations of Martin's Axiom, $\Box^*_{\o_1}$, $\lnot\wCC$, and a forcing axiom failure}\markleft{Generalizations of Martin's Axiom, $\Box^*_{\o_1}$, $\wCC$, and a forcing axiom failure}

\begin{abstract}
We prove that the forcing axiom $\MA^{1.5}_{\al_2}(\mbox{stratified})$ implies $\Box_{\omega_1, \o_1}$. 
Using this implication, we show that the forcing axiom $\MM_{\aleph_2}(\aleph_2\mbox{-c.c.})$ is inconsistent. We also derive weak Chang's Conjecture from $\MA^{1.5}_{\al_2}(\mbox{stratified})$ and use this second implication to give another proof of the inconsistency of $\MM_{\aleph_2}(\aleph_2\mbox{-c.c.})$.  \end{abstract}

\section{Introduction}

Jensen's square principle $\Box_\k$ at an infinite successor cardinal $\k^+$, as well as its natural weakenings $\Box_{\k, {<}\l}$, are very well studied principles in combinatorial set theory that can be naturally viewed as incompactness or anti-reflection principles. For example, and as is well-known, $\Box_\k$ implies that for every stationary $S\sub\k^+$ there is a stationary $S'\sub S$ which does not reflect (i.e., such that $S'\cap\a$ is non-stationary for every $\a<\k^+$). We recall the definition of these weakenings of $\Box_\k$. 

\begin{definition}
Given a cardinal $\k\geq\o_1$ and a cardinal $\l\leq\k^+$, $\Box_{\k, {<}\l}$ holds if and only if there is a sequence $(\mtcl C_\a\,;\,\a\in\Lim(\k^+))$ with the following properties.
\begin{enumerate}
\item For every $\a\in \Lim(\k^+)$, 
\begin{enumerate}
\item $\mtcl C_\a$ is a collection of clubs of $\a$,
\item $\ot(C)\leq\k$ for every $C\in\mtcl C_\a$, and
\item $|\mtcl C_\a|<\l$.
\end{enumerate}
\item (Coherence) For every $\a\in\Lim(\k^+)$, $C\in\mtcl C_\a$, and $\bar\a<\a$, if $\bar\a$ is a limit point if $C$, then $C\cap\bar\a\in \mtcl C_{\bar\a}$.  
\end{enumerate} 
\end{definition}

In the above definition, and throughout the paper, given an ordinal $\mu$, $\Lim(\mu)$ will denote the set of nonzero limit ordinals in $\m$. 

We also write $\Box_{\k, \l}$ for $\Box_{\k, {<}\l^+}$. We will be focusing our attention on the weakest form of $\Box_{\k, \l}$ for $\k=\o_1$, i.e., on $\Box_{\o_1, \o_1}$. This principle is also called \emph{weak square} and, by a well-known result of Jensen, is equivalent to the existence of a special $\al_2$-Aronszajn tree. 

Forcing axioms are reflection principles, between forcing extensions and the ground model, which tend to imply failures of square. For example, a well-known of Todor\v{c}evi\'{c} (\cite{Todorcevic}) is that the Proper Forcing Axiom ($\PFA$) implies that $\Box_\k$ fails for every uncountable cardinal $\k$.   

Given a cardinal $\k$ and a class $\mtcl K$ of partial orders, we denote the forcing axiom for $\mtcl K$ with respect to families of $\k$-many dense sets by $\FA_\k(\mtcl K)$. This is the assertion that for every $\mtbb P\in\mtcl K$ and every sequence $(D_i\,:\,i<\k)$ of dense subsets of $\mtbb P$ there is a filter $G\sub\mtbb P$ such that $G\cap D_i\neq\emptyset$ for all $i<\k$. 
We will say that \emph{$\FA_\k(\mtcl K)$ is the forcing axiom for $\mtcl K$ at $\k$}.

In \cite{genMA}, a certain forcing axiom $\MA^{1.5}_\k$ extending Martin's Axiom at $\k$ (i.e., $\MA_\k$) is introduced. It is proved in that paper that if $\CH$ holds, $\k$ is any closed enough cardinal, and a suitable diamond principle holds, then $\MA^{1.5}_\k$ holds in some cardinal-preserving forcing extension. The first half of  the present paper is devoted to the study of the effects, on square principles at $\o_2$, of forcing axioms which are natural mild strengthenings of $\MA^{1.5}_{\al_2}$. The main result in this part of the paper is that one such axiom, which we will refer to as $\MA^{1.5}_{\o_2}(\mbox{stratified})$, implies $\Box_{\o_1, \o_1}$.
Moreover, $\MA^{1.5}_{\o_2}(\mbox{stratified})$, and in fact $\MA^{1.5}_\k(\mbox{stratified})$ for any given $\k$ fixed in advance, is consistent by essentially the same proof as in \cite{genMA}. 

Certainly, weak forcing axioms such as $\MA_\k$, for any $\k$, are compatible with square principles: any cardinal-preserving forcing extension of a model with squares will preserve them, and $\MA_\k$ can always be forced by c.c.c.\ forcing. What is perhaps surprising, given that strong forcing axioms at $\o_1$ like $\PFA$ prohibit squares, is the fact that strong enough forcing axioms at $\o_2$, like $\MA^{1.5}_{\o_2}(\mbox{stratified})$, outright imply square principles at $\o_2$. 

Similar results have been obtained by Neeman. One of the results in \cite{Neeman} is that a forcing axiom $\MA^{1.5}_{\o_2}(U)$, whose definition involves a certain parameter $U\sub [\o_2]^{\al_0}$, implies both $\Box_{\o_1, {<}\o}$ and the following strengthening $\Box_{\o_1, \o}^{\textrm{ta}}$ of $\Box_{\o_1, \o}$:\footnote{On the other hand, the definition of $\MA^{1.5}_{\o_2}(\mbox{stratified})$ is parameter-free.} 
Given cardinals $\l\leq\k$ such that $\k\geq\o_1$, $\Box^{\text{ta}}_{\k, \l}$ holds if and only if there is a $\Box_{\k, \l}$-sequence $(\mtcl C_\a\,:\,\a\in\Lim(\k^+))$ such that for every $\a\in \Lim(\k^+)$ and for all $C$, $C'\in\mtcl C_\a$, $C$ and $C'$ agree on a tail, i.e., there is some $\b<\a$ such that $C\setminus\b=C'\setminus\b$.\footnote{The superscript \text{ta} stands for `tail agreement'.}

Neeman also points out in \cite{Neeman} that both $\Box_{\o_1, \o}$ and $\Box_{\o_1, \o}^{\textrm{ta}}$ follow from some of his strong high analogues of $\PFA$. On a related vein,\footnote{One important aspect in which Sakai's result is different from both Neeman's results and the first main result in the present paper is that Sakai's theorem involves strong forcing axiom at $\o_1$, whereas the others are implications from forcing axioms at $\o_2$.} Sakai shows in \cite{Sakai} that Martin's Maximum proves that $\Box^{\text p}_{\o_1}$ (i.e, partial square at $\o_2$) holds and that this is not the case for $\PFA$.

Given a set $N$ such that $N\cap\o_1\in \o_1$, we denote this ordinal by $\d_N$ and call it \emph{the height of $N$}. Given a sequence $\vec e=(e_\a\,:\,\a\in \o_2)$, where $e_\a:|\a|\into \a$ is a bijection for each $\a<\o_2$, we say that a set $N$ is \emph{closed under $\vec e$} if $e_\a\restr\x+1\in N$ whenever $\a\in \o_2\cap N$ and $\x\in |\a|\cap N$. We will be using the following well-known fact repeatedly, sometimes without mention.

\begin{fact}\label{agreement} Suppose $\vec e=(e_\a\,:\,\a\in \o_2)$, where $e_\a:|\a|\into \a$ is a bijection for each $\a<\o_2$, and suppose $N_0$ and $N_1$ are countable submodels of $H(\o_2)$ closed under $\vec e$ and such that $\d_{N_0}\leq\d_{N_1}$. Then $N_0\cap\a\sub N_1$ for every $\a\in N_0\cap N_1\cap\o_2$.
\end{fact}

\begin{proof}
Given any  $\bar\a\in N_0\cap\a$ there is some $\xi\in N_0\cap |\a|$ such that $e_\a(\xi)=\bar\alpha$. But since $\a$ and $\xi$ are both members of $N_1$ as $|\a|\leq\o_1$, we also have that $\bar\alpha=e_\a(\xi)\in N_1$.\end{proof}

\begin{corollary}\label{agreement-0}
Suppose $\vec e=(e_\a\,:\,\a\in \o_2)$, where $e_\a:|\a|\into \a$ is a bijection for each $\a<\o_2$, and suppose $N_0$ and $N_1$ are countable submodels of $H(\o_2)$ closed under $\vec e$ of the same height. Then $N_0\cap N_1\cap\o_2$ is an initial segment of both $N_0\cap\o_2$ and $N_1\cap\o_2$.
\end{corollary}

As already mentioned, \cite{genMA} introduces a strengthening $\MA^{1.5}_\k$ of $\MA_\k$, for any given cardinal $\k$. $\MA^{1.5}_\k$ is the forcing axiom $$\FA_\k(\{\mtbb P\,:\,\mtbb P\mbox{ has the }\al_{1.5}\mbox{-c.c.}\}),$$ where having the $\al_{1.5}$-c.c.\ is defined as follows.\footnote{One can also define $\MA^{1.5}_{{<}\k}$ as $\MA^{1.5}_\l$ for all $\l<\k$.}

\begin{definition} (\cite{genMA}) A partial order $\mtbb P$ has the $\al_{1.5}$-c.c.\ if for every large enough cardinal $\t$ (i.e., every cardinal $\t$ such that $\mtbb P\in H(\t)$) there is a club $E$ of $[H(\t)]^{\al_0}$ such that for every finite $\mtcl N\sub E$ and every $N_0\in\mtcl N$, if $N_0$ has minimum height within $\mtcl N$, then for every $p_0\in N_0\cap\mtbb P$ there is some extension $p\in\mtbb P$ of $p_0$ such that $p$ is $(N, \mtbb P)$-generic for all $N\in\mtcl N$. \end{definition}

It is easy to see (s.\ \cite{genMA}) that every partial order with the countable chain condition has the $\al_{1.5}$-c.c.\ and every partial order with the $\al_{1.5}$-c.c.\ is proper and has the $\al_2$-c.c.\footnote{Hence the $\al_{1.5}$ notation.}

We can naturally strengthen $\MA^{1.5}_\k$ by restricting the definition to finite families $\mtcl N\sub E$ with some given nice structural property $P$. We may then refer to the corresponding enlargement of $$\{\mtbb P\,:\,\mtbb P\mbox{ has the }\al_{1.5}\mbox{-c.c.}\}$$ as the class of partial orders with the $\al_{1.5}$-c.c.\ with respect to finite families with property $P$, and may denote the corresponding forcing axiom by $\MA^{1.5}_\k(P)$.

The following notion of stratified family seems to give rise to a particularly useful form of $\MA^{1.5}_\k$. 

\begin{definition}\label{stratified}
A collection $\mtcl N$ of countable elementary submodels of $H(\t)$, for some infinite cardinal $\t$, is \emph{stratified} in case for all $N_0$, $N_1\in\mtcl N$, if $\d_{N_0}<\d_{N_1}$, then in fact $\ot(N_0\cap\o_2)<\d_{N_1}$. 
\end{definition}

Correspondingly, we say that a forcing notion $\mtbb P$ \emph{has the $\al_{1.5}$-c.c.\ with respect to finite stratified families of models} iff for every infinite cardinal $\t$ such that $\mtbb P\in H(\t)$ there is a club $E\sub [H(\t)]^{\al_0}$ such that for every finite stratified $\mtcl N\sub E$, if $p\in N_0\cap\mtbb P$, where $N_0\in\mtcl N$ is of minimal height within $\mtcl N$, then there is an extension $p^*$ of $p$ in $\mtbb P$ such that $p^*$ is $(N, \mtbb P)$-generic for every $N\in \mtcl N$. Clearly, every partial order with the $\al_{1.5}$-c.c.\ also has the $\al_{1.5}$-c.c.\ with respect to finite stratified families of models. Also, given a cardinal $\k$, we write $\MA^{1.5}_\k(\mbox{stratified})$ to denote $\FA_\k(\mtcl K)$, where $\mtcl K$ is the class of partial orders $\mtbb P$ such that $\mtbb P$ has the $\al_{1.5}$-c.c.\ with respect to finite stratified families of models. We then have that 
$\MA^{1.5}_\k(\mbox{stratified})$ implies $\MA^{1.5}_\k$. 

The following proposition extends the aforementioned fact that every forcing with the $\al_{1.5}$-c.c.\ is proper and has the $\al_2$-c.c.

\begin{proposition}\label{al15-2} If a forcing notion has the $\al_{1.5}$-c.c.\ with respect to finite stratified families of models, then it is proper and has the $\al_2$-c.c.\end{proposition}

\begin{proof} Suppose $\mtbb P$ is a forcing notion with the $\al_{1.5}$-c.c.\ with respect to finite stratified families of models. Let $\t$ be a cardinal such that $\mtbb P\in H(\t)$ and let $E\sub [H(\t)]^{\al_0}$ be a club witnessing, for $H(\t)$, that $\mtcl P$ has the $\al_{1.5}$-c.c.\ with respect to finite stratified families of models. Given any $N\in E$, $\{N\}$ is trivially stratified, and therefore for every $p\in \mtbb P\cap N$ there is an $(N, \mtbb P)$-generic extension of $p$. This shows that $\mtbb P$ is proper. 

To prove that $\mtbb P$ has the $\al_2$-chain condition let us assume, towards a contradiction, that there is a maximal antichain $A$ of $\mtbb P$ such that $|A|\geq\al_2$, and let $(p_i\,:\,i<\l)$ be a one-to-one enumeration of $A$, for some $\l\geq\o_2$. Let $M$ be an elementary submodel of some large enough $H(\chi)$ such that
\begin{enumerate}
\item  $E$, $A\in M$ and
\item $|M|=\al_1$.
\end{enumerate}

Let $i_0\in \o_2\setminus M$ and let $N\preccurlyeq H(\t)$ be countable and such that $p_{i_0}$, $E$, $M\in N$.
By correctness of $M$ we may find $i_1\in \o_2\cap M$ for which there is some $N'\in E\cap M$ such that $\d_{N'}=\d_N$ and $p_{i_1}\in N'$.
Indeed, the existence of such an $N'$ is expressed by a true sentence, as witnessed by $i_0$ and $N$, with $\d_N$, $E$ and $(p_i\,:\,i<\l)$ as parameters.

We note that $\mtcl N=\{N, N'\}$ is a stratified family of members of $E$ as $\d_N=\d_{N'}$. It follows, since $p_{i_0}\in N$ and $\d_{N}=\d_{N'}$, that we may find an $(N', \mtbb P)$-generic condition $p$ extending $p_{i_0}$. Then there must be condition $p'$ extending $p$ and extending some $\bar p\in A\cap N'$. But that is impossible since $A$ is an antichain and $\bar p\neq p_{i_0}$ as $N'\sub M$.
\end{proof}

Essentially the same forcing construction from \cite{genMA} showing the consistency of $\MA^{1.5}_{{<}\k}$, for any given closed enough $\k$, can be used to prove the following theorem.\footnote{The main point is that all relevant systems of models coming up in the proof from \cite{genMA} turn out in fact to be stratified.} 

\begin{theorem}\label{mainthm000} ($\textsf{CH}$) Let $\k\geq\o_2$ be a regular cardinal such that $\m^{\al_0}< \kappa$ for all $\m < \k$ and $\diamondsuit(\{\a<\k\,:\,\cf(\a)\geq\o_1\})$ holds. Then there is a proper forcing notion $\mtcl P$ of size $\k$ with the $\al_2$-chain condition such that the following statements hold in the generic extension by $\mtcl P$.

\begin{enumerate}

\it[(1)] $2^{\al_0}=\k$
   
\it[(2)] For every $\l<\k$, $\MA^{1.5}_\l(\mbox{stratified})$
\end{enumerate}
\end{theorem}

Let us write $\PFA_{\al_2}(\aleph_2\mbox{-c.c.})$ to denote $\FA_{\al_2}(\mtcl K)$, where $\mtcl K$ is the class of proper forcing notions with the $\al_2$-chain condition. 

Theorem \ref{mainthm000}, as well as other similar strengthenings of the main result from \cite{genMA}, motivate the following question.\footnote{This question is also motivated by the main result from \cite{Asp-Golsh}, to the effect that it is consistent, for arbitrary choice of $\kappa$, that $\FA_\k(\{\mtbb P\,:\, \mtbb P\mbox{ proper}, |\mtbb P|=\al_1\})$ holds. It is worth pointing out that the proof of this theorem is very different from the proof of the main result from \cite{genMA}.}

\begin{question} Is $\PFA_{\al_2}(\aleph_2\mbox{-c.c.})$ consistent?
\end{question}

We are not able to answer this question. However, in Section \ref{s2}
we will show that $\MM_{\al_2}(\aleph_2\mbox{-c.c.})$, a natural strengthening of $\PFA_{\al_2}(\aleph_2\mbox{-c.c.})$, is in fact inconsistent. $\MM_{\al_2}(\aleph_2\mbox{-c.c.})$ is  $\FA_{\al_2}(\mtcl K^*)$, where $\mtcl K^*$ is the class of forcing notions that both preserve stationary subsets of $\o_1$ and have the $\al_2$-chain condition.

One of the ingredients of this proof will be the fact that $\MM_{\al_2}(\aleph_2\mbox{-c.c.})$ implies $\Box_{\o_1, \o_1}$ (since $\MM_{\al_2}(\aleph_2\mbox{-c.c.})$ extends $\MA^{1.5}_{\al_2}(\mbox{stratified})$).

There are other results in the literature dealing with failures of forcing axioms at $\al_2$ or above. In this respect we single out the following theorem of Shelah (\cite{Shelah}), extended by the main result in Section \ref{s2}. 

\begin{theorem} (Shelah) Given any regular cardinal $\l>\o_1$, $\FA_\l(\mtcl K_\l)$ is false, where $\mtcl K_\l$ is the class of forcing notions preserving all stationary subsets of $\m$ for every uncountable regular cardinal $\m\leq\l$. 
\end{theorem}

We will actually give two proofs of the failure of $\MM_{\al_2}(\aleph_2\mbox{-c.c.})$. The first one, in Section \ref{s2}, we have already referred to. The second proof, in Section \ref{s4}, will make use of a consequence of $\MA^{1.5}_{\al_2}(\mbox{stratified})$ regarding canonical functions. We recall that if $\a<\o_2$ is a nonzero ordinal and $\p:\o_1\into\a$ is a surjection, the function $g_\a:\o_1\into \o_1$ defined by letting $g_\a(\n)=\ot(\p``\n)$ is called a, or the, \emph{canonical function for $\a$}. The use the definite article when referring to canonical functions for a given $\a$ is justified by the following obvious observation, which in particular implies that $g_\a$ is uniquely determined modulo clubs.

\begin{fact} Given $\a<\o_2$ and given surjections $\p_0$, $\p_1:\o_1\into \a$ there is a club of $\n<\o_1$ such that $\p_0``\n=\pi_1``\n$. \end{fact}

Another standard fact about the canonical function for $\a$ is that it represents the ordinal $\a$ in every generic ultrapower of $V$ obtained from forcing with $\mtcl P(\o_1)/\NS_{\o_1}$. In other words, if $g$ is a canonical function for $\a$, then $\mtcl P(\o_1)/\NS_{\o_1}$ forces that, letting $M=((^{\o_1^V}V)\cap V)/\dot G$ be the generic ultrapower of $V$ obtained from $\dot G$, the set of $M$-ordinals below the class $[g]_{\dot G}$ of $g$ in $M$ is well-ordered in order type $\a$.  

\emph{Club-bounding by canonical functions}, $\CB$, is the statement that every function $f:\o_1\into \o_1$ is bounded on a club by the canonical function of some nonzero $\a<\o_2$. $\CB$ is a weakening of $\NS_{\o_1}$ being saturated.

A natural weakening of $\CB$ is \emph{weak Chang's Conjecture}, $\wCC$, which is the statement that for every function $f:\o_1\into  \o_1$ there is some $\a<\o_2$ such that $\{\n<\o_1\,:\, f(\n)<g(\n)\}$ is stationary for every canonical function $g$ for $\a$ (s.\ \cite{donder-levinski}). Martin's Maximum implies the saturation of $\NS_{\o_1}$  (\cite{FMS}), and hence also $\CB$. On the other hand, it is not difficult to see that not even $\wCC$ follows from $\PFA$ (s.\ e.g.\ \cite{Asp} for strong forms of this non-implication).

In Section \ref{s3} we will prove that $\MA^{1.5}_{\al_2}(\mbox{stratified})$ implies $\lnot\wCC$. Using this implication, in Section \ref{s4} we will give another proof of the inconsistency of $\MM_{1.5}(\aleph_2\mbox{-c.c.})$. Specifically, assuming this forcing axiom, we will produce a sequence $(f_n)_{n<\o}$ of functions from $\o_1$ to $\o_1$ such that for each $n$, $f_n$ dominates $f_{n+1}$ on a club. This is of course impossible as then we are able to find an infinite decreasing sequence of ordinals. 

The rest of the paper is structured as follows. 
In the following section we prove that $\MA^{1.5}_{\al_2}(\mbox{stratified})$ implies $\Box_{\o_1, \o_1}$. Using this result, in Section \ref{s2} we prove that $\MM_{\al_2}(\aleph_2\mbox{-c.c.})$ is false. In Section \ref{s3} we prove that $\MA^{1.5}_{\al_2}(\mbox{stratified})$ implies $\lnot\wCC$. Finally, in Section \ref{s4}, using the implication in Section \ref{s3}, we give another proof of the inconsistency of $\MM_{\al_2}(\aleph_2\mbox{-c.c.})$.

Throughout the paper, we will write $S^2_1$ for $\{\a<\o_2\,:\,\cf(\a)=\o_1\}$, and sometimes $S^2_0$ for $\{\a<\o_2\,:\,\cf(\a)=\o\}$. All undefined pieces of notation are hopefully standard and may be found in \cite{JECH} or \cite{KUNEN}.

\section{$\MA^{1.5}_{\al_2}(\mbox{stratified})$ implies $\Box_{\o_1, \o_1}$}\label{s1}

The main theorem in this section is the following.

\begin{theorem}\label{thm1} $\MA^{1.5}_{\al_2}(\mbox{stratified})$ implies $\Box_{\o_1, \o_1}$.
 \end{theorem}

Let $\vec e=(e_\a\,:\, \a<\o_2)$ be such that $e_\a:|\a|\into\a$ is a bijection for each $\a<\o_2$. 
We define the following forcing notion $\mtcl P$.

Conditions in $\mtcl P$ are triples $$p=(h^p, i^p, \mtcl N_p)$$ with the following properties.

\begin{enumerate}

\item $h^p$ is a function such that $\dom(h_p)\in  [\Lim(\o_2)\times\o_1\times (\o_1+1)]^{{<}\o}$ and such that for each $(\a,  \n, \tau)\in \dom(h^p)$, $h^p(\a, \n, \tau)\sub\tau\times\a$ is a finite function which can be extended to a strictly increasing and continuous function $f:\tau\into\a$ with range cofinal in $\a$.
\item For every $(\a, \n, \tau)\in\dom(h^p)$:
\begin{enumerate}
\item  if $\cf(\a)=\o_1$, then $\tau=\o_1$ and $\n=0$;
\item  if $\cf(\a)=\o$, then $\tau\in\Lim(\o_1)$. 
\end{enumerate}
\item For every $\a$ and $\n$ there is at most one $\tau$ such that $(\a, \n, \tau)\in\dom(h^p)$.

\item $i^p$ is a function whose domain is the set of triples $(\a, \n, \bar\tau)$ such that $(\a, \n, \tau)\in\dom(h^p)$ for some $\tau$ and $\bar\tau\in\dom(h^p(\a, \n, \tau))\cap\Lim(\o_1)$, and $i^p(\a, \n, \bar\tau)\in\o_1$ for each $(\a, \n, \bar\tau)\in\dom(i^p)$.

\item The following holds for each $(\a,  \n, \tau)\in \dom(h^p)$.
\begin{enumerate}

\item For every limit ordinal $\bar\tau\in\dom(h^p(\a, \n,  \tau))$, 
$$(h^p(\a, \n, \tau)(\bar\tau), i^p(\a, \n, \bar\tau), \bar\tau)\in \dom(h^p)$$ and $$h^p(h^p(\a, \n, \tau)(\bar\tau),  i^p(\a, \n, \bar\tau), \bar\tau)=h^p(\a, \n, \tau)\restr\bar\tau$$
\end{enumerate}

\item $\mtcl N$ is a finite stratified collection of countable elementary submodels of $(H(\o_2); \in)$ closed under $\vec e$.

\item For every $N\in\mtcl N$ and every $(\a, \n, \tau)\in\dom(h^p)$ such that $\a$, $\n\in N$:
\begin{enumerate}
\item $\tau\in N$;
\item  $h^p(\a, \n, \tau)\restr N\sub N$;\footnote{In particular, if $\cf(\a)=\o$, then $h^p(\a, \n, \tau)\in N$.}
\item If $\cf(\a)=\o_1$, then $$\d_N\in\dom(h^p(\a, \n, \o_1))$$ and $$h^p(\a, \n, \o_1)(\d_N)=\ssup(N\cap\a)$$
\item For every $\bar\tau\in\dom(h^p(\a, \n, \tau))\cap\Lim(\o_1)\cap N$, $i^p(\a, \n, \bar\tau)\in N$.
\end{enumerate}

\end{enumerate}

Given $\mtcl P$-conditions $p_0$ and $p_1$, $p_1$ extends $p_0$ if and only if:
\begin{enumerate}
\item $\dom(h^{p_0})\sub\dom(h^{p_1})$;
\item for every $(\a, \n, \tau)\in \dom(h^{p_0})$, 
\begin{enumerate}
\item $h^{p_0}(\a, \n, \tau)\sub h^{p_1}(\a, \n, \tau)$, and
\item $i^{p_1}(\a, \n, \bar\tau)=i^{p_0}(\a, \n, \bar\tau)$ for each $\bar\tau\in\dom(h^{p_0}(\a, \n, \tau))\cap\Lim(\o_1)$.
\end{enumerate}
\item $\mtcl N_{p_0}\sub \mtcl N_{p_1}$.
\end{enumerate}

Given $p\in\mtcl P$, we denote $\{\a\in S^2_1\,:\,(\a, \n, \tau)\in\dom(h^p)\mbox{ for some }\n,\,\tau\}$ by $X_p$.

\begin{lemma}\label{aleph15} $\mtcl P$ has the $\al_{1.5}$-c.c.\ with respect to finite stratified families of models.
\end{lemma} 

\begin{proof} Let $\t$ be such that $\mtcl P\in H(\t)$ and let $\mtcl N^*$ be a finite stratified collection of countable elementary submodels of $H(\t)$ containing $\vec e$.
We may assume that for each $N^*\in\mtcl N^*$, $N^*=\bigcup_{i<\d_{N^*}}N^*_\nu$, where $(N^*_\nu)_{\nu<\d_{N^*}}$ is a continuous $\in$-chain of countable elementary submodels of $H(\t)$ containing $\vec e$.
 In fact, if there is any cardinal $\chi>\t$ such that $N^*$ is of the form $N^{**}\cap H(\t)$ for a countable $N^{**}\elsub H(\chi)$ with $\vec e$, $\t\in N^{**}$, then $N^*$ is a continuous $\in$-chain of countable elementary submodels of $H(\t)$ as above. To see this, let $(x_n)_{n<\o}$ be an enumeration of $N^*$ and let $(\b_n)_{n<\o}$ be a strictly increasing sequence converging to $\d_{N^*}$. Then, by correctness of $N^{**}$, we may build a sequence $(\vec N^*_n)_{n<\o}$ of members of $N^{**}$ such that 
 \begin{enumerate}
 \item for each $n$, $\vec N^*_n = (N^*_{n, i})_{i\leq \b_n}$ is a continuous $\in$-chain of length $\b_n+1$ consisting of countable elementary submodesl of $H(\t)$ containing $\vec e$ and $x_n$;
 \item  $\vec N^*_{n+1}$ extends $\vec N^*_n$. \end{enumerate}
 $\bigcup_{n<\o}\vec N^*_n$ is then as desired.
 
Let $\mtcl N=\{N^*\cap H(\o_2)\,:\, N^*\in\mtcl N^*\}$. 
Let also $N_0\in\mtcl N$ be of minimal height and let $p_0\in N_0$ be a $\mtcl P$-condition. Given any $\a\in X_{p_0}$, let $(\a(k)\,:\,k<m_\a)$ be the strictly increasing enumeration of $$\{\ssup(N\cap\a)\,:\, N\in\mtcl N,\,\a\in N\}$$ and, for every $k<m_\a$, let $\d^\a_k=\d_N$ for any $N\in\mtcl N$ such that $\a\in N$ and $\ssup(N\cap\a)=\a(k)$. Note that by Fact \ref{agreement}, $\d^\a_k$ is well-defined for every $k<m_\a$ (i.e., $\d^\a_\k$ is independent from the choice of $N$ as long as $\a\in N$ and $\ssup(N\cap\a)=\a(k)$) as in fact $N\cap\a=N'\cap\a$ whenever $N$, $N'\in\mtcl N$ are such that $\a\in N\cap N'$ and $\d_N=\d_{N'}$. For each $\a\in X_{p_0}$ and $k<m_\a$, let $i(\a, k)\in\d^\a_k\setminus \range(i^{p_0})$ be such that $\d_N<i(\a,  k)$ for each $N\in\mtcl N$ with $\d_N<\d^\a_k$.

In order to prove the lemma, it suffices to show that $$p^*=(h^{p^*}, i^{p^*}, \mtcl N_{p_0}\cup\mtcl N)$$ is an $(N^*, \mtcl P)$-generic condition for each $N^*\in\mtcl N^*$, where $$\dom(h^{p^*})=\dom(h^{p_0})\cup \{(\a(k), i(\a, k), \d^\a_k)\,:\,\a\in X_{p_0},\,k<m_\a\}$$ and where for each $(\a, \n, \tau)\in\dom(h^{p^*})$:

\begin{enumerate}
\item if $(\a, \n, \tau)\in\dom(h^{p_0})$ and $\cf(\a)=\o$, then 
\begin{enumerate} 
\item $h^{p^*}(\a, \n, \tau)=h^{p_0}(\a, \n, \tau)$ and 
\item $i^{p^*}(\a, \n, \bar\tau)=i^{p_0}(\a, \n, \bar\tau)$ for each $\bar\tau\in\dom(h^{p_0}(\a, \n, \tau))\cap\Lim(\o_1)$; 
\end{enumerate} 
\item  if $(\a, \n, \tau)\in\dom(h^{p_0})$ and $\cf(\a)=\o_1$,\footnote{In this case of course $\n=0$ and $\tau=\o_1$.} then
\begin{enumerate}
\item $h^{p^*}(\a, \n, \tau)=h^{p_0}(\a, \n, \tau)\cup\{(\d^\a_k, \a(k))\,:\,k<m_\a\}$,
\item  $i^{p^*}(\a, \n, \bar\tau)=i^{p_0}(\a, \n, \bar\tau)$ for every $\bar\tau\in\dom(h^{p_0}(\a, \n, \tau))\cap\Lim(\o_1)$, and
\item $i^{p^*}(\a, \n, \d^\a_k)=i(\a, k)$ for each $k<m_\a$;
\end{enumerate}
\item if $\a\in X_{p_0}$ and $k<m_\a$, then 
\begin{enumerate}
\item $h^{p^*}(\a(k), i(\a, k), \d^\a_k)=h^{p_0}(\a, 0, \o_1)\cup\{(\d^\a_{j}, \a(j))\,:\, j<k\}$,
\item $i^{p^*}(\a(k), i(\a, k), \bar\tau)=i^{p_0}(\a, 0, \bar\tau)$ for every limit ordinal $\bar\tau\in\dom(h^{p_0}(\a, 0, \o_1))$, and
\item $i^{p^*}(\a(k), i(\a, k), \d^\a_j)=i(\a, j)$ for each $j<k$.
\end{enumerate}
\end{enumerate}

\begin{claim}\label{cl0}
If $\a_0<\a_1$ are such that $\a_0$, $\a_1\in X_{p_0}$, then $\a_0(k_0)<\a_0< \a_1(k_1)$ for all $k_0<m_{\a_0}$ and $k_1<m_{\a_1}$.
\end{claim}

\begin{proof} 
The inequality $\a_0(k_0)<\a_0$ is immediate given that $\a_0(k_0)=\ssup(M\cap\a_0)$ for some countable $M$. Also, we note that if $N\in\mtcl N$ is such that $\a_1\in N$ and $\a_1(k_1)=\ssup(N\cap\a_1)$, then $\a_0\in N$ by Fact \ref{agreement} since $\a_1\in N_0\cap N$ and $\d_{N_0}\leq\d_N$. Hence $\a_0<\ssup(N\cap\a_1)=\a_1(k_1)$.
\end{proof}

\begin{claim}\label{cl1}
For every $N\in\mtcl N$, $\a\in X_{p_0}$ and $k<m_\a$, if $\a(k)$, $i(\a, k)\in N$ and $j<k$, then $\a(j)\in N$.
\end{claim}

\begin{proof}
Since $i(\a, k)\in N$, we have that $\d_N\geq\d^\a_k$. But $\a(j)\in M\cap\a(k)$ for every $M\in\mtcl N$ such that $\a\in M$ and $\d_M=\d^\a_k$ (since $\a(j)=\ssup(e_\a``\d^\a_j)$), and $M\cap\a(k)\sub N\cap\a(\k)$ for every such $M$, where the inclusion follows from Fact \ref{agreement} since $\d_N\geq\d^\a_k$.
\end{proof}

The proof of the following claim is essentially the same.

\begin{claim}\label{cl2}
For all $N\in\mtcl N$ and $\a\in X_{p_0}$, if $\a\in N$, then $h^{p_0}(\a, 0, \o_1)\in N$ and $i^{p_0}(\a, 0, \bar\tau)\in N$ for every $\bar\tau\in\dom(h^{p_0}(\a, 0, \o_1))\cap\Lim(\o_1)$.
\end{claim}

Using the above three claims together with Fact \ref{agreement}, one can easily verify that $p^*$ is a $\mtcl P$-condition, and it obviously extends $p_0$. Let now $N^*\in\mtcl N^*$ and let us prove that $p^*$ is $(N^*, \mtcl P)$-generic. For this, let $D\in N^*$ be an open and dense subset of $\mtcl P$ and let $p\in D$ extend $p^*$. We will prove that there is a condition $r\in D\cap N^*$ compatible with $p$.

Let $(N^*_\n)_{\n<\d_{N^*}}$ be a continuous $\in$-chain of countable elementary submodels of $H(\t)$ containing $\vec e$ such that $N^*=\bigcup_{\n<\d_{N^*}}N^*_\n$. Since $\mtcl N_p$ is stratified and $N^*\cap H(\o_2)\in\mtcl N_p$, we may find some $\n_0<\d_{N^*}$ such that 
\begin{enumerate}
\item $(h^p\cup i^p)\cap N^*\sub N^*_{\n_0}$,
\item there is some $\eta\in N^*_{\n_0}\cap\o_2$ such that $[\eta,\,\o_2)\cap N^*_{\n_0}\cap N=\emptyset$ for every $N\in\mtcl N_p$ with $\d_N<\d_{N^*_{\n_0}}$, and
\item for every $\a\notin N^*_{\n_0}$ such that $(\a, \n, \tau)\in \dom(h^p)$ for some $\n$, $\tau$, and such that $\a^*=\min((N^*_{\n_0}\cap\o_2)\setminus\a)$ exists, there is some $\eta_\a\in N^*_{\n_0}\cap\a^*$ with $[\eta_\a,\,\a)\cap N^*_{\n_0}\cap N=\emptyset$ for every $N\in\mtcl N_p$ such that $\d_N<\d_{N^*_{\n_0}}$. 
\end{enumerate}

Given a $\mtcl P$-condition $q$, let $\mtbb M(q)$ be a structure with universe $$\mtcl U_q:=(\bigcup\{\{\a, \n, \tau, \x, \b\}\,:\, (\a, \n, \tau)\in\dom(h^q),\, (\x, \b)\in h^q(\a, \n, \tau)\})\cup i^q$$ 
coding $h^q$ and $i^q$ in some fixed canonical way. 

Let us denote $N^*_{\n_0}$ by $N^+$. Let $R=\mtcl U_p\cap N^+$. Working in $N^+$ we may find a condition $r\in D$ such that $\mtcl U_p\cap N^+\sub\mtcl U_r$, $X_p\cap N^+\sub X_r$, and for which there is an isomorphism $$\pi:\mtbb M(p)\into\mtbb M(r)$$ which is the identity on $\mtcl U_p\cap\mtcl U_r$ and is such that the following holds for each $(\a, \n, \tau)\in \dom(h^p)$:

\begin{enumerate}
\item if $\a\geq\ssup(N^+\cap\o_2)$, then $\pi(\a)>\eta$;
\item if $\a\notin N^+$ and $\a^*=\min((N^+\cap\o_2)\setminus\a)$ exists, then $\pi(\a)>\eta_\a$;
\item if $\a\in N^+$ but $\n\notin N^+$, then $\pi(\n)>\d_N$ for each $N\in\mtcl N_p$ such that $\d_N<\d_{N^+}$;
\item $\mtcl N_p\cup\mtcl N_r$ is stratified.
 \end{enumerate}
 
Such an $r$ can indeed be found in $N^+$ since the existence of a condition with the properties above is a true statement, as witnessed by $p$, which can be expressed over $H(\theta)$ by a sentence with parameters in $N^+$.

In order to finish the proof it suffices to show that $p$ and $r$ can be amalgamated into a condition $p'\in\mtcl P$. This condition $p'$ can be obtained as $p'=(h^{p'}, i^{p'}, \mtcl N_p\cup\mtcl N_r)$ by the following construction, very similar to that of $p^*$ from $p_0$. 

For every $\a\in X_r$, let $(\a(k)\,:\,k<m_\a)$ be the strictly increasing enumeration of $$\{\ssup(N\cap\a)\,:\, N\in\mtcl N_p,\,\a\in N\}$$ and, for every $k<m_\a$, let $\d^\a_k=\d_N$ for any $N\in\mtcl N_p$ such that $\a\in N$ and $\ssup(N\cap\a)=\a(k)$. As in the construction of $p^*$ from $p_0$, each $\d^\a_k$ is well-defined. For each $\a\in X_r$ and $k<m_\a$, let $i(\a, k)\in\d^\a_k\setminus\range(i^p)$ be such that $\d_N<i(\d,  k)$ for each $N\in\mtcl N_p$ with $\d_N<\d^\a_k$.

We define $h^{p'}$ and $i^{p'}$ by letting $h^{p'}$ be a function with $$\dom(h^{p'})=\dom(h^p)\cup\dom(h^r)\cup \{(\a(k), i(\a, k), \d^\a_k)\,:\,\a\in X_r,\,k<m_\a\}$$ and making the following definitions for each $(\a, \n, \tau)\in\dom(h^{p'})$ (where, given a condition $t\in\mtcl P$ and a tuple $(\a, \n, \tau)\notin \dom(h^t)$, we define $h^t(\a, \n, \tau)=\emptyset$ if $(\a, \n, \tau)\notin \dom(h^t)$, and similarly with $i^t$ in place of $h^t$):

\begin{enumerate}
\item if $(\a, \n, \tau)\in\dom(h^p)\cup\dom(h^r)$ and $\cf(\a)=\o$, then 
\begin{enumerate} 
\item $h^{p'}(\a, \n, \tau)=h^p(\a, \n, \tau)\cup h^r(\a, \n, \tau)$ and 
\item for each $\bar\tau$ in $(\dom(h^p(\a, \n, \tau))\cup\dom(h^r(\a, \n, \tau)))\cap\Lim(\o_1)$, $i^{p'}(\a, \n, \bar\tau)=i^p(\a, \n, \bar\tau)\cup i^r(\a, \n, \bar\tau)$; 
\end{enumerate} 
\item if $(\a, \n, \tau)\in\dom(h^p)$ and $\cf(\a)=\o_1$,\footnote{In which case of course $\n=0$ and $\tau=\o_1$.} then 
\begin{enumerate}
\item $h^{p'}(\a, \n, \tau)=h^p(\a, \n, \tau)\cup h^r(\a, \n, \tau)$,
\item $i^{p'}(\a, \n, \bar\tau)=i^p(\a, \n, \bar\tau)$ for every $\bar\tau\in\dom(h^p(\a, \n, \tau))\cap\Lim(\o_1)$, and
\item $i^{p'}(\a, \n, \bar\tau)=i^r(\a, \n, \bar\tau)$ for  every $\bar\tau\in\dom(h^r(\a, \n, \tau))\cap\Lim(\o_1)$;
\end{enumerate}
\item if $(\a, \n, \tau)\in\dom(h^r)$ and $\cf(\a)=\o_1$,\footnote{Once again, in this case $\n=0$ and $\tau=\o_1$.} then
\begin{enumerate}
\item $h^{p'}(\a, \n, \tau)=h^r(\a, \n, \tau)\cup\{(\d^\a_k, \a(k))\,:\,k<m_\a\}$,
\item  $i^{p'}(\a, \n, \bar\tau)=i^r(\a, \n, \bar\tau)$ for every $\bar\tau\in\dom(h^r(\a, \n, \tau))\cap\Lim(\o_1)$, and
\item $i^{p'}(\a, \n, \d^\a_k)=i(\a, k)$ for each $k<m_\a$;
\end{enumerate}
\item if $\a\in X_r$ and $k<m_\a$, then 
\begin{enumerate}
\item $h^{p'}(\a(k), i(\a, k), \d^\a_k)=h^r(\a, 0, \o_1)\cup\{(\d^\a_j, \a(j))\,:\, j<k\}$,
\item $i^{p'}(\a(k), i(\a, k), \bar\tau)=i^r(\a, 0, \bar\tau)$ for every limit ordinal $\bar\tau\in \dom(h^r(\a, 0, \o_1))$, and
\item $i^{p'}(\a(k), i(\a, k), \d^\a_j)=i(\a, j)$ for each $j<k$.
\end{enumerate}
\end{enumerate}




The choice of $\eta$ and of $\eta_\a$, for $\x\in X_p\setminus N^+$ such that $\min((N^+\cap\o_2)\setminus\a)$ exists, together with the way $r$ has been fixed, immediately yields the following.

\begin{claim}\label{cl3}
For every $N\in\mtcl N_p$ and every $\a\in X_r\setminus X_p$, if $\a\in N$, then $\d_N\geq\d_{N^+}$. 
\end{claim}

Using Claim \ref{cl3}, we can prove the following versions of Claims \ref{cl0} and \ref{cl2}.

\begin{claim}\label{cl4}
If $\a_0<\a_1$ are such that $\a_0$, $\a_1\in X_r$, then $\a_0(k_0)<\a_0< \a_1(k_1)$ for all $k_0<m_{\a_0}$ and $k_1<m_{\a_1}$.
 \end{claim}
 
\begin{claim}\label{cl6}
For all $N\in\mtcl N_p$ and $\a\in X_r$, if $\a\in N$, then $h^r(\a, 0, \o_1)\in N$ and $i^r(\a, 0, \bar\tau)\in N$ for every $\bar\tau\in\dom(h^r(\a, 0, \o_1))\cap\Lim(\o_1)$.
\end{claim}

We also have the following counterpart of Claim \ref{cl1}, proved in exactly the same way.

 \begin{claim}\label{cl5}
For every $N\in\mtcl N_p$, $\a\in X_r$ and $k<m_\a$, if $\a(k)$, $i(\a, k)\in N$ and $j<k$, then $\a(j)\in N$.
\end{claim}

Using Claims \ref{cl4}, \ref{cl6} and \ref{cl5}, together with Fact \ref{agreement} and the particular choice of $r$, it is not difficult to verify that $p'$ is a condition in $\mtcl P$, which finishes the proof of the lemma since then $p'$ of course extends both $p$ and $r$. 
\end{proof}

We will need the following four density lemmas.

\begin{lemma}\label{density3} For every $\a<\o_1$ of countable cofinality and every $p\in\mtcl P$ there is a condition $p'\in\mtcl P$ extending $p$ and such that $(\a, 0, \o)\in\dom(h^{p'})$. 
\end{lemma}

\begin{proof}
We simply  let $p'=(h^p\cup\{((\a, 0, \o), \emptyset)\}, i^p, \mtcl N_p)$.
\end{proof}

\begin{lemma}\label{density0} For every $\a\in S^2_1$ and every $p\in\mtcl P$ there is a condition $p'\in\mtcl P$ extending $p$ and such that $\a\in X_{p'}$.
 \end{lemma}

\begin{proof}
We may obviously assume $\a\notin X_p$. We may also assume that $\a\in N$ for some $N\in\mtcl N$ as the conclusion in the other case is immediate. 
Let $(\a(k)\,:\,k<m_\a)$ be the strictly increasing enumeration of $$\{\ssup(N\cap\a)\,:\, N\in\mtcl N_p,\,\a\in N\}$$ and, for every $k<m_\a$, let $\d^\a_k=\d_N$ for any $N\in\mtcl N$ such that $\a\in N$ and $\ssup(N\cap\a)=\a(k)$. As usual, using Fact \ref{agreement} we have that each $\d^\a_k$ is well-defined. For each $\a\in X_r$ and $k<m_\a$, let $i(\a, k)\in\d^\a_k\setminus\range(i^p)$ be such that $\d_N<i(\d,  k)$ for each $N\in\mtcl N_p$ with $\d_N<\d^\a_k$.

We can now easily verify that the following is a condition $p'\in\mtcl P$ as required: $p'=(h^{p'}, i^{p'}, \mtcl N_p)$, where $$\dom(h^{p'})=\dom(h^p)\cup \{(\a, 0, \o_1)\} \cup\{(\a(k), i(\a, k), \d^\a_k)\,:\,k<m_\a\}$$ and where for each $(\a, \n, \tau)\in\dom(h^{p'})$:

\begin{enumerate}
\item if $(\a, \n, \tau)\in\dom(h^p)$, then 
\begin{enumerate} 
\item $h^{p'}(\a, \n, \tau)=h^p(\a, \n, \tau)$ and 
\item $i^{p'}(\a, \n, \bar\tau)=i^p(\a, \n, \bar\tau)$ for each $\bar\tau\in\dom(h^p(\a, \n, \tau))\cap\Lim(\o_1)$; 
\end{enumerate} 
\item $h^{p'}(\a, 0, \o_1)=\{(\d^\a_k, \a(i))\,:\,i<m_\a\}$ and $i^{p'}(\a, 0, \d^\a_k)=i(\a, k)$ for each $k<m_\a$;
\item for each $k<m_\a$, 
\begin{enumerate}
\item $h^{p'}(\a(k), i(\a, k), \d^\a_k)=\{\d^\a_j\,:\,j<k\}$ and
\item $i^{p'}(\a(k), i(\a, k), \d^\a_j)=i(\a, j)$ for each $j<k$.
\end{enumerate}
\end{enumerate}
\end{proof}

Lemmas \ref{density1} and \ref{density2} are also easy.

\begin{lemma}\label{density1} For every $p\in\mtcl P$, $\a\in X_p$, and every $\n<\o_1$ there is a condition $p'\in\mtcl P$ extending $p$ and such that $\n\in \dom(h^{p'}(\a, 0, \o_1))$.  \end{lemma}

\begin{lemma}\label{density2} For every $p\in\mtcl P$, $\a\in X_p$, every nonzero limit ordinal $\d\in \dom(h^p(\a, 0, \o_1))$, and every $\eta<h^p(\a, 0, \o_1)(\d)$ there is a condition $p'\in\mtcl P$ extending $p$ together with some $\n\in \dom(h^{p'}(\a, 0, \o_1))\cap\d$ such that $h^{p'}(\a, 0, \o_1)(\nu)>\eta$.   
\end{lemma}

Given a $\mtcl P$-generic filter $G$, a limit ordinal $\a<\o_2$, and $\n<\o_1$, we define $C^G_{\a, \n}$ as $$\bigcup\{\range(h^p(\a, \n, \tau))\,:\, p\in G,\,(\a, \n, \tau)\in\dom(h^p)\mbox{ for some }\tau\}$$ Let also $$\mtcl  C^G_\a=\{C^G_{\a, \n}\,:\,\n<\o_1,\,C^G_{\a, \n}\neq\emptyset\}$$ 

We immediately obtain the following corollary from Lemmas \ref{aleph15}, the density lemmas \ref{density0}--\ref{density2}, and the definition of $\mtcl P$.

\begin{corollary}\label{cor00} If $G$ is a $\mtcl P$-generic filter over $V$,  then $$(\mtcl C^G_\a\,:\,\a\in\Lim(\o_2))$$ is a $\Box_{\o_1, \o_1}$-sequence. 
\end{corollary}

Corollary \ref{cor00} yields the following.

\begin{corollary}\label{cor1} $\MA^{1.5}_{\al_2}(\mbox{stratified})$ implies $\Box_{\o_1, \o_1}$.\end{corollary} 
 
 At this point, the following question suggests itself.

\begin{question} Does $\MA^{1.5}_{\al_2}(\mbox{stratified})$ imply $\Box_{\o_1, \o}$? \end{question}

It is proved in \cite{Neeman} that $\MA^{1.5}_\k$, for any given $\k$, is consistent with $\lnot\Box_{\o_1, \o}$. 

\begin{question}
Does $\MA^{1.5}_{\al_2}$ imply $\Box_{\o_1, \o_1}$? \end{question} 
 
 Finally, the following corollary is an immediate consequence of Corollary \ref{cor00}.
 
  \begin{corollary}
  $\ZFC$ proves that there is a poset $\mtcl P$ such that \
  \begin{enumerate}
  \item $\mtcl P$ is proper,
  \item $\mtcl P$ has the $\al_2$-c.c., and
  \item $\mtcl P$ forces weak square.
  \end{enumerate}
  \end{corollary}

\section{$\MM_{\aleph_2}(\aleph_2\mbox{-c.c.})$ is false}\label{s2}

Given a set $S$ of ordinals and a set $X$, let us denote by $\Unif_{S, X}$ the statement that for every sequence $(f_\a\,:\,\a\in S)$ of colourings $f_\a\sub\a\times X$ such that $\dom(f_\a)$ is a club of $\a$ there is a function $H:\bigcup S\into X$ such that for every $\a\in S$, $$\{\x\in \dom(f_\a)\,:\, f_\a(\x)=H(\x)\}$$ contains a club of $\a$.

Shelah proves the following theorem in (\cite{PIF}, Appendix, Chapter 3).

\begin{theorem}\label{Shelah-thm} (Shelah) $\Unif_{S^2_1, 2}$ is false. 
\end{theorem}

We can also define a natural weakening $\Unif^{\,\text{c}}_{S, X}$ of $\Unif_{S, X}$ by restricting to sequences $(f_\a\,:\,\a\in S)$ of constant colourings (i.e., for every $\a\in S$, $f_\a$ is a constant function).\footnote{The superscript c  is for `constant'.} It is immediate to see that for any $S\sub \Ord$ and any set $X$, $\Unif^{\,\text{c}}_{S, X}$ can be equivalently stated as the assertion that for every function $F:S\into X$ there is a function $H:\bigcup S\into X$ with the property that for every $\a\in S$ there is a club $C\sub \a$ of $\a$ such that $H(\x)=F(\a)$ for every $\x\in C$. We will say that $H$ \emph{uniformizes $F$ mod.\ clubs}.



The following is implicit in (\cite{PIF}, Appendix, Chapter 3).

\begin{theorem}\label{unif} (Shelah) If $S\sub S^2_1$ is stationary and $\Unif^{\,\text{c}}_{S, 2}$ holds, then $\CH$ holds as well. 
\end{theorem}  

\begin{proof}
$\Unif^{\,\text{c}}_{S, 2}$ clearly implies $\Unif^{\,\text{c}}_{S, \mtbb R}$: Given $F:S\into\, ^\o2$, let $F_n:S\into 2$ be defined by $F_n(\a)=(F(\a))(n)$ (for each $n<\o$). Applying $\Unif^{\,\text{c}}_{S, 2}$ to each $F_n$ we obtain functions $H_n:S\into 2$ and clubs $D^n_\a\sub\a$,  for $\a\in S$ and $n<\o$, such that $H_n(\x)= F_n(\a)$ for all $\x\in D^n_\a$. But then, if we define $H:S\into\, ^\o2$ by letting $H(\x)=(H_n(\x)\,:\,n<\o)$, it follows that $H$ uniformizes $F$ mod.\ clubs as witnessed by the clubs $D_\a$, for $\a\in S$, where $D_\a=\bigcap_n D^n_\a$. 

Thus, if $2^{\al_0}\geq\al_2$ and $\Unif^{\,\text{c}}_{S, 2}$ holds, then $\Unif^{\,\text{c}}_{S, \o_2}$ holds as well. 
Now suppose $\Unif^{\,\text{c}}_{S, 2}$ holds and $2^{\al_0}\geq\al_2$. Letting $F$ be the identity function on $S$, we apply $\Unif^{\,\text{c}}_{S, \o_2}$ to $F$ and get a corresponding uniformizing function $H:\o_2\into \o_2$ and clubs $D_\a\sub \a$ for $\a\in S$. Since $S$ is stationary, we may find $\a\in S$ closed under $H$. But now we reach a contradiction since there is obviously no club $D\sub \a$ such that $H(\x)=F(\a)=\a$ for all $\x\in D$.\footnote{There is obviously not even any nonempty $D\sub\a$ like that.} 
\end{proof}

\begin{remark}\label{rmk2}
Given a club-sequence $\vec C=(C_\a\,:\,\a\in S)$ such that $\ot(C_\a)=\cf(\a)$ for each $\a\in S$, we can define the following strengthening $\Unif^{\,\text{c}, \vec C,\, \text{cbd}}_{S, 2}$ of $\Unif^{\,\text{c}}_{S, 2}$: $\Unif^{\,\text{c}, \vec C,\, \text{cbd}}_{S, 2}$ is the statement that for every function $F:S\into 2$ there is a function $H:\ssup(S)\into 2$ such that for every $\a\in S$, $$\{\x\in C_\a\,:\, H(\x)=F(\a)\}$$ is co-bounded in $\a$.\footnote{We say that $H$ uniformizes $F$ on $\vec C$ modulo co-bounded sets.}

If $\CH$ holds and $\vec C=(C_\a\,:\,\a\in S)$ is as above, $\Unif^{\,\text{c}, \vec C,\, \text{cbd}}_{S^2_1, 2}$ can be forced by a $\s$-closed and $\al_2$-c.c.\ forcing, obtained as the direct limit of a long enough countable support iteration of $\s$-closed forcing notions with the $\al_2$-c.c.
At any given stage of the iteration, the corresponding iterand is the forcing $\mtcl Q_{\vec C, F}$ for adding a uniformizing function on $\vec C$ mod.\ co-bounded sets, for some given colouring $F:S\into 2$:  A condition in $\mtcl Q_{\vec C, F}$ is a function $q=(b^q_\a\,:\,\a\in Z_q)$, for some countable $Z_q\sub S^2_1$, such that $b^q_\a<\a$ for each $\a\in Z_q$, and such that $\ssup(\dom(C_{\a'})\cap\a)<b^q_{\a'}$ for all $\a<\a'$ in $Z_q$ with $F(\a)\neq F(\a')$.
The extension relation is reversed inclusion. 
\end{remark}

\begin{question}
Is the dependence on a fixed club-sequence in the consistency proof in Remark \ref{rmk2} necessary? In other words, is the following strengthening of $\Unif^{\,\text{c}, \vec C_\ast,\, \text{cbd}}_{S, 2}$, for a fixed club-sequence $\vec C_*=(C_\a,:\,\a\in S^2_1)$ with $\ot(C_\a)=\o_1$ for each $\a$, consistent? Suppose $\vec C=(C_\a\,:\,\a\in S)$ is a club-sequence such that $\ot(C_\a)=\o_1$ for each $\a\in S$. Then  for every function $F:S\into 2$ there is a function $H:\ssup(S)\into 2$ such that for every $\a\in S$, $$\{\x\in C_\a\,:\, H(\x)=F(\a)\}$$ is co-bounded in $\a$.
\end{question}

\begin{remark}
The statement that $\Unif^{\,\text{c}}_{S, 2}$ holds for every stationary $S\sub S^2_1$ is not equivalent to $\CH$ as, for example, the assumption that $\diamondsuit(S)$ holds for every stationary $S\sub S^2_1$ implies $\lnot\Unif^{\,\text{c}}_{S, 2}$ for every such $S$: Suppose $(A_\a\,:\,\a\in S)$ is a $\diamondsuit$-sequence and let $F:S\into 2$ be such that for every $\a\in S$,  $F(\a)=1-i$ if $A_\a$ codes a function $H_\a:\a\into 2$ and there are club-many $\x\in \a$ such that $H_\a(\x)=i$. It is easy to see that no function $H:\o_2\into 2$ can uniformize $F$ mod.\ clubs.
\end{remark}




Given a class $\mtcl K$ of countable models, let us say that a proper forcing $\mtbb P$ is \emph{proper with respect to $\mtcl K$} in case for every cardinal $\theta$ such that $\mtbb P\in H(\theta)$ there is a club $D\sub [H(\t)]^{\al_0}$ such that for every $N\in D\cap\mtcl K$ and every condition $p\in \mtbb P\cap N$ there is an extension $p^*\in\mtbb P$ of $p$ which is $(N, \mtbb P)$-generic.

Given a cardinal $\t$, a set $\mtcl S\sub [H(\t)]^{\al_0}$ is a \emph{projective stationary subset of $H(\t)$} in case for every stationary $S\sub\o_1$ and every club $D$ of $[H(\t)]^{\al_0}$ there is some $N\in \mtcl S\cap D$ such that $\d_N\in S$. The following proposition is standard.

\begin{proposition}\label{proj-stat} 
Let $\mtcl K$ be a class of models such that $\mtcl K\cap [H(\t)]^{\al_0}$ is a projective stationary subset of $[H(\t)]^{\al_0}$ for every cardinal $\t>\o_1$ such that $\mtbb P\in H(\t)$. Let $\mtbb P$ be a forcing notion which is proper with respect to $\mtcl K$. Then $\mtbb P$ preserves stationary subsets of $\o_1$. 
\end{proposition}

\begin{proof}
Let $\dot C$ be a $\mtbb P$-name for a club of $\o_1^V$, let $S\sub\o_1$ be stationary, and let $p\in\mtbb P$. Let $\t$ be large enough and, using the projective stationarity of $\mtcl K\cap H(\t)$, let $N\prec H(\t)$ be countable and such that $\mtbb P$, $\dot C$, $p\in N$ and $\d_N\in S$. Let $p^*$ be an $(N,\mtbb P)$-generic condition stronger than $p$. Then $p^*$ forces that $\d_N\in S$ is a limit of ordinals in $\dot C$ and therefore, since $\dot C$ is a $\mtbb P$-name for a closed set, that $\d_N\in\dot C$. 
\end{proof}

Given a cardinal $\k$, $\MM_\k(\aleph_2\mbox{-c.c.})$ denotes $\FA_\k(\Gamma)$, where $\Gamma$ is the class of all posets $\mtbb P$ such that 
\begin{itemize}
\item $\mtbb P$ preserves stationary subsets of $\o_1$ and 
\item $\mtbb P$ has the $\al_2$-c.c. 
\end{itemize}

The rest of this section is devoted to proving the following theorem.

\begin{theorem}\label{mainthm} $\MM_{\aleph_2}(\aleph_2\mbox{-c.c.})$ is false. \end{theorem}

Let us assume, towards a contradiction, that $\MM_{\aleph_2}(\aleph_2\mbox{-c.c.})$ holds. In particular $\FA_{\al_2}(^{{<}\o}2)$ holds and therefore $\CH$ fails.\footnote{In fact $2^{\al_0}\geq\al_3$.} Let $S=S^2_1$. It follows, by Theorem \ref{unif}, that there is a function $F:S\into 2$ for which there is no function $H:\o_2\into 2$ uniformizing $F$ mod.\ clubs. 

Since $\MA^{1.5}_{\al_2}(\mbox{stratified})$ also holds, we may fix a $\Box_{\o_1, \o_1}$-sequence $\vec{\mtcl C}=(\mtcl C_\a\,:\,\a\in\Lim(\o_2))$ (by Theorem \ref{thm1}). Let also $\vec e=(e_\a\,:\,\a<\o_2)$ be such that $e_\a:|\a|\into\a$ is a bijection for each $\a<\o_2$. 

Let $\mtcl K^{\vec e}_{\vec{\mtcl C}}$ be the class of countable models $N$ such that $N\cap \o_2=\bigcup_{\g\in C}e_\g``\d_N$ for some $C\in\mtcl C_\a$, where $\a=\ssup(N\cap\o_2)$.

The following is quite standard.

\begin{claim}\label{cl-proj-stat} For every cardinal $\t>\o_1$, $\mtcl K^{\vec e}_{\vec{\mtcl C}}\cap H(\t)$ is a projective stationary subset of $[H(\t)]^{\al_0}$.
\end{claim}

\begin{proof} Suppose $D$ is a club of $[H(\t)]^{\al_0}$ and $S\sub\o_1$ is stationary. Let $f:[\o_2]^{{<}\o}\into \o_2$ be a finitary function such that for every $X\in [\o_2]^{\al_0}$, if $f``[X]^{{<}\o}\sub X$, then $X=N\cap \o_2$ for some $N\in D$. Let $\a\in S^2_0$ be such that $\o_1<\a$ and $f``[\a]^{{<}\o}\sub\a$ and let $C\in\mtcl C_\a$. But now, since $E=\{M\cap\a\,:\, M\prec (H(\o_2); \in, \vec e, C)\}$ contains a club of $[\a]^{\al_0}$, we may pick some $X\in E$ closed under $f$ and such that $\d_X\in S$, and if $N\in D$ is such that $N\cap\o_2=X$, then $N$ will be a member of $\mtcl K^{\vec e}_{\vec{\mtcl C}}$ such that $\d_N\in S$.\footnote{The choice of $\a$ being of countable cofinality is inessential. We could have taken $\a$ of cofinality $\o_1$, considered  $E=\{M\cap\a\,:\, M\prec (H(\o_2); \in, \vec e, \vec{\mtcl C})\}$, and continued the argument using the coherence of $\vec{\mtcl C}$.}
\end{proof}

We will show that there is a forcing notion $\mtcl Q$ which is proper with respect to $\mtcl K^{\vec e}_{\vec{\mtcl C}}$, has the $\al_2$-c.c., and forces the existence of a function $H:\o_2\into 2$ uniformizing $F$ mod.\ clubs. This will yield a contradiction since then $\mtcl Q$ will preserve stationary subsets of $\o_1$ by Proposition \ref{proj-stat} and Claim \ref{cl-proj-stat}, and so the existence of such a function $H$ will follow from an application of $\FA_{\al_2}(\{\mtcl Q\})$.


For each $\a\in\Lim(\o_2)$, let us fix an enumeration $(C_{\a, \n}\,:\,\n<\o_1)$ of $\mtcl C_\a$. If $\cf(\a)=\o_1$, we may of course take $(C_{\a, \n}\,:\,\n<\o_1)$ to be constant. Also, given a set $X$, we will write $\cl(X)$ to denote $X\cup\overline{X\cap\Ord}$, where $\overline{X\cap\Ord}$ denotes the closure of $X\cap\Ord$ in the order topology.\footnote{We stress that $X$ need not be a set of ordinals.} 

Let us say that a family $\mtcl N$ of countable models is \emph{$\vec{\mtcl C}$-stratified} in case the following holds.

\begin{enumerate}
\item $\mtcl N\sub\mtcl K^{\vec e}_{\vec{\mtcl C}}$
\item For all $N_0$, $N_1\in\mtcl N$, if $\d_{N_0}=\d_{N_1}$ but $N_0\cap\o_2\neq N_1\cap \o_2$, then 
\begin{enumerate}
\item $\a_i:=\min((N_i\cap \o_2)\setminus N_{1-i})$ exists for each $i\in 2$, 
\item $\cf(\a_0)=\cf(\a_1)=\o_1$, and 
\item there is no ordinal  $\a$ above $\ssup(N_0\cap N_1\cap\o_2)$ such that $\a\in\cl(N_0\cap\o_2)\cap\cl(N_1\cap\o_2)$. 
\end{enumerate}
\item For all $N_0$, $N_1\in\mtcl N$, if $\d_{N_0}<\d_{N_1}$, then $$\a:=\max(\cl(N_0\cap\o_2)\cap \cl(N_1\cap\o_2))$$ exists, $\a\in N_1$, and there is some $\n<\d_{N_1}$ such that $$N_0\cap\a=\bigcup_{\g\in C_{\a, \n}}e_\g``\d_{N_0}.$$ 
 \end{enumerate}


The following simple remark will be quite useful.

\begin{remark}\label{r0}
Suppose $\mtcl N$ is a $\vec{\mtcl C}$-stratified family of models, $\bar\a<\o_2$, $N_0$, $N_1\in\mtcl N$, and $\a_0\in N_0\cap S$ and $\a_1\in N_1\cap S$ are such that $$\ssup(N_0\cap\a_0)=\ssup(N_1\cap\a_1)=\bar\a$$ Then $\d_{N_0}=\d_{N_1}$. Hence, if $\a_0\neq\a_1$, then $\a_0=\min((N_0\cap\o_2)\setminus N_1)$ and  $\a_1=\min((N_1\cap\o_2)\setminus N_0)$.
\end{remark}

Let us also say that a $\vec{\mtcl C}$-stratified family $\mtcl N$ of models is \emph{compatible with $F$} in case for all $N_0$, $N_1\in\mtcl N$, if $\d_{N_0}=\d_{N_1}$, $N_0\cap \o_2\neq N_1\cap\o_2$, and $\a_i=\min((N_i\cap \o_2)\setminus N_{1-i})$ for each $i\in 2$, then $F(\a_0)=F(\a_1)$.

We define $\mtcl Q$ to be the forcing notion consisting of ordered pairs $$q=((\mtcl I^q_\a\,:\, \a\in X_q), \mtcl N_q)$$ with the following properties.

\begin{enumerate}
\item $X_q\in [S]^{{<}\o}$
\item For every $\a\in X_q$, $\mtcl I^q_\a$ is a finite collection of pairwise disjoint intervals of the form $[\gamma_0,\,\gamma_1)$ with $\g_0<\g_1<\a$.
\item For all $\a_0$, $\a_1\in X_q$, if $F(\a_0)\neq F(\a_1)$, then $\min(I)\neq \min(I')$ for all $I\in\mtcl I^q_{\a_0}$ and $I'\in \mtcl I^q_{\a_1}$.
\item $\mtcl N_q$ is a finite family of countable elementary submodels of the structure $(H(\o_2); \in, \vec e, (C_{\a, \n}\,:\, \a\in\Lim(\o_2),\, \n<\o_1))$ which is $\vec{\mtcl C}$-stratified and compatible with $F$.
\item The following are equivalent for every $\a\in X_q$ and every $\b<\a$.
\begin{enumerate}
\item $\b=\min(I)$ for some $I\in\mtcl I^q_\a$.
\item $\b=\ssup(N\cap\a)$ for some $N\in\mtcl N_q$ such that $\a\in N$.
\end{enumerate}
\end{enumerate}

Given conditions $q_0$, $q_1\in\mtcl Q$, $q_1$ extends $q_0$ iff 
\begin{enumerate}
\item $X_{q_0}\sub X_{q_1}$,
\item for every $\a\in X_{q_0}$ and every $I\in\mtcl I^{q_0}_\a$ there is some (necessarily unique) $I'\in\mtcl I^{q_1}_\a$ such that $\min(I')=\min(I)$ and $\ssup(I')\geq\ssup(I)$, and
\item $\mtcl N_{q_0}\sub\mtcl N_{q_1}$
\end{enumerate}

We will use the two following density lemmas.

\begin{lemma}\label{dens0}
For every $\mtcl Q$-condition $q$ and every $\a\in S$ there is some $q^*\in\mtcl Q$ extending $q$ and such that $\a\in X_{q^*}$. 
\end{lemma}

\begin{proof}
We may obviously assume that $\a\notin X_q$. Let $$\mtcl I=\{\{\ssup(N\cap\a)\}\,:\, N\in\mtcl N_q,\,\a\in N\}$$ Then $$q^*:=((\mtcl I^q_\b\,:\,\b\in X_q)\cup\{(\a, \mtcl I)\}, \mtcl N_q)$$ is a condition in $\mtcl Q$ as desired. To see this, let $\bar\a=\ssup(N\cap\a)$ for some $N\in\mtcl N_q$ with $\a\in N$ and suppose, towards a contradiction, that there is some $\a'\in X_q$ and some $N'\in\mtcl N_q$ such that $\a'\in N'$, $\ssup(N\cap\a')=\bar\a$, and $F(\a')\neq F(\a)$.\footnote{It is not difficult to check that this is the only way $q^*$ could fail to be a $\mtcl Q$-condition.} By $\vec{\mtcl C}$-stratification of $\mtcl N_q$ and Remark \ref{r0} we have that $\d_N=\d_{N'}$, $\a=\min((N\cap\o_2)\setminus N')$, and  $\a'=\min((N'\cap\o_2)\setminus N)$. But then $F(\a)=F(\a')$ since $\mtcl N_q$ is compatible with $F$, which is a contradiction. 
\end{proof}

\begin{lemma}\label{dens2} For all $q\in\mtcl Q$, $\a\in X_q$, and $\eta<\a$ there is some extension $q^*\in\mtcl Q$ together with some $I\in\mtcl I^{q^*}_\a$ such that $\min(I)>\eta$. \end{lemma}

\begin{proof}
Let $N$ be a sufficiently correct elementary submodel of $H(\o_2)$ containing $q$ and $\eta$. We build $q^*$ as $$q^*=((\mtcl I^{q^*}_\b\,:\,\b\in X_q),  \mtcl N_q\cup\{N\}),$$ where $$\mtcl I^{q^*}_\b=\mtcl I^{q}_\b\cup\{\{\ssup(N\cap\b)\}\}$$ for each $\b\in X_q$. 
It is clear that $q^*$ is  condition in $\mtcl Q$ stronger than $q$. Also, $\eta<\ssup(N\cap\a)$ since $\eta\in N$.
\end{proof}

The main properness lemma is now the following.

\begin{lemma}\label{Q-proper} Let $\t$ be a cardinal such that $\mtcl Q\in H(\t)$ and let $M^0$ and $M^1$ be countable elementary submodels of $H(\t)$ of the same height such that $F$, $\vec C$, $\vec e\in M^0\cap M^1$ and $\{M^0, M^1\}$ is a $\vec{\mtcl C}$-stratified family\footnote{I.e., conditions (1) and (2) in the definition of $\vec{\mtcl C}$-stratified family hold for $M_0$ and $M_1$.} compatible with $F$. Then for every $q_0\in\mtcl Q\cap M^0$ there is an extension $q^*\in \mtcl Q$ of $q_0$ such that $q^*$ is $(M^i, \mtcl Q)$-generic for $i=0$, $1$. 
\end{lemma}

\begin{proof}
Let $$\mtcl N=\{M^0\cap H(\o_2), M^1\cap H(\o_2)\}$$ and for every $\a\in X_{q_0}$ let $\r_\a=\ssup(M^0\cap\a)$. 

The proof will be complete once we show that
 $$q^*=((\mtcl I^{q^*}_\a\,:\,\a\in X_{q_0}), \mtcl N_{q_0}\cup\mtcl N)$$ is an $(M^i, \mtcl Q)$-generic condition for $i=0$, $1$, 
where $\mtcl I^{q^*}_\a=\mtcl I^q_\a\cup \{\{\r_\a\}\}$.\footnote{Of course $\{\r_\a\}=[\r_\a,\,\r_\a+1)$.}

\begin{claim}\label{cl8}
 $\mtcl N_{q_0}\cup\mtcl N$ is $\vec{\mtcl C}$-stratified.
 \end{claim}
 
 \begin{proof}
 Since $\mtcl N_{q_0}$ and $\mtcl N$ are both $\vec{\mtcl C}$-stratified and $\mtcl N_{q_0}\sub \mtcl K^{\vec e}_{\vec{\mtcl C}}\cap M^0$, it suffices to show that if $N\in\mtcl N_{q_0}$, then $$\a:=\max(\cl(N\cap\o_2)\cap\cl(M^1\cap\o_2))$$ exists, $\a\in M^1$, and there is some $\n<\d_{M^1}$ such that $N\cap\a = \bigcup_{\g\in C_{\a, \n}} e_\g``\d_N$. 
 
 Suppose first that $N\cap\o_2\sub M^1$ and let $\a=\ssup(N\cap\o_2)$. Since $N\in M^0\cap\mtcl K^{\vec e}_{\vec{\mtcl C}}$, we then have that there is some $\n<\d_{M^0}=\d_{M^1}$ such that $N\cap\o_2=\bigcup_{\g\in C_{\a, \n}}e_\g``\d_N$. By $\vec{\mtcl C}$-stratification of $\mtcl N$ and since $\a\in M^0$, $\a$ cannot be $\ssup(M^0\cap M^1\cap\o_2)$, so it must be the case that $\a\in M^1$. Since $\a=\max(\cl(N\cap\o_2)\cap\cl(M^1\cap\o_2))$, we are done in this case. 
 
 Suppose now that $\b=\min((N\cap\o_2)\setminus M^1)$ exists. Since $\b\in M^0$, again by $\vec{\mtcl C}$-stratification of $\mtcl N$ it follows that $\cf(\b)=\o_1$. Then $C_{\b, 0}\in N$, from which we get that if $\a=\ssup(N\cap\b)$, then $N\cap\a=\bigcup_{\g\in C_{\b, 0}\cap\a}e_\g``\d_N$. Then, by coherence of $\vec{\mtcl C}$, we have that $C_{\b, 0}\cap\a\in \mtcl C_\a$. Again since $N\in M^0$, $ C_{\b, 0}\cap\a=C_{\a, \n}$ for some $\n<\d_{M^0}=\d_{M^1}$. Using once again the $\vec{\mtcl C}$-stratification of $\mtcl N$ it follows that $\a\in M^1$, which finishes the proof in this case since $\a= \max(\cl(N\cap\o_2)\cap\cl(M^1\cap\o_2))$. 
 \end{proof}

Also, since $\mtcl N_{q_0}$ and $\mtcl N$ are both compatible with $F$, so is $\mtcl N_{q_0}\cup\mtcl N$. In addition, for every $\a\in X_{q_0}$ and every $I\in\mtcl I^{q_0}_\a$, $\min(I)<\r_\a$ and $\r_\a\notin M^0$, and for all $\a<\a'$ in $X_{q_0}$, $\r_\a<\a<\r_{\a'}$. It thus easily follows that $q^*$ is a condition in $\mtcl Q$. Since $q^*$ of course extends $q_0$, it suffices to prove that $q^*$ is $(M^i, \mtcl Q)$-generic for each $i=0$, $1$. For this, suppose $D\in M^i$ is an open and dense subset of $\mtcl Q$ and $q$ is an extension of $q^*$ in $D$. We will find a condition in $D\cap M^i$ compatible with $q$.   

Let $\Delta=\{\d_N\,:\,N\in\mtcl N_q\}\cap\d_{M^i}$. Let us note that, by $\vec{\mtcl C}$-stratification of $\mtcl N_q$ and $M^i\in\mtcl N_q$, $$R_q=\{N\cap \o_2\cap M^i\,:\,N\in\mtcl N_q,\,\d_N\in\Delta\}\in M^i.$$ Using this, and by a reflection argument as in the proof of Lemma \ref{aleph15}, we may find in $M^i$ a condition $r\in D$ such that
$$q':=((\mtcl I^q_\a\oplus\mtcl I^r_\a\,:\,\a\in X_q\cup X_r), \mtcl N_q\cup\mtcl N_r)\in\mtcl Q,$$ where $\mtcl I^q_\a\oplus\mtcl I^r_\a$ is defined as follows for each $\a\in X_q\cup X_r$.
 
 \begin{enumerate}
 \item If $\a\in X_q\setminus X_r$, then $\mtcl I^q_\a\oplus\mtcl I^r_\a=\mtcl I^q_\a$.
 \item If $\a\in X_r\setminus X_q$, then $$\mtcl I^q_\a\oplus\mtcl I^r_\a=\mtcl I^r_\a\cup\{\{\ssup(N\cap\a)\}\,:\,N\in\mtcl N_q,\,\a\in N\}$$
 \item If $\a\in X_q\cap X_r$, then $\mtcl I^q_\a\oplus\mtcl I^r_\a$ is the unique set $\mtcl I$ of pairwise disjoint intervals with $$\{\min(I)\,:\,I\in\mtcl I\}=\{\min(I)\,:\,I\in\mtcl I^q_\a\cup\mtcl I^r_\a\}$$ such that for every $I\in\mtcl I$, if $\g_0=\min(I)$, then 
 \begin{enumerate}
 \item $\ssup(I)=\ssup(I_0)$ in case $I_0\in\mtcl I^q_\a$, $\min(I_0)=\g_0$, and there is no $J\in\mtcl I^r_\a$ such that $\min(J)=\g_0$; 
 \item $\ssup(I)=\ssup(I_1)$ in case $I_1\in\mtcl I^r_\a$, $\min(I_1)=\g_0$, and there is no $J\in\mtcl I^q_\a$ such that $\min(J)=\g_0$; 
 \item $\ssup(I)=\max\{\ssup(I_0), \ssup(I_1)\}$ in case $I_0\in\mtcl I^q_\a$,  $I_1\in\mtcl I^r_\a$,  $\min(I_0)=\g_0$, and $\min(I_1)=\g_0$.
\end{enumerate}
\end{enumerate}

More specifically, we find $r\in D\cap M^i$ with the following properties.

\begin{enumerate}
\item For all $\a\in X_r$, $I\in\mtcl I^r_\a$, $\a\in X_q$, and $I'\in\mtcl I^q_{\a'}$, if $\min(I)=\min(I')$, then $F(\a)=F(\a')$. 
\item For every $N\in\mtcl N_r$ such that $\d_N\in\Delta$ there is some $\sub$-maximal  $X\in\ R_q$ such that $X$ is an initial segment of $N\cap\o_2$. Moreover, $X$ is a proper initial segment of $N\cap\o_2$ if and only if there if there is some $N'\in\mtcl N_q$ such that $X$ is a proper initial segment of $N'\cap\o_2$, in which case, for every such $N'$, if $$\a_0=\min((N\cap\o_2)\setminus X)$$ and $$\a_1=\min((N'\cap\o_2)\setminus X),$$ then $\cf(\a_0)=\cf(\a_1)=\o_1$ and $F(\a_0)=F(\a_1)$. Also, for every $X\in R_q$ there is some $N\in\mtcl N_r$ such that $\d_N\in\Delta$ and such that $X$ is an initial segment of $N\cap\o_2$.
\end{enumerate}

We can indeed find such an $r\in M^i$, by correctness of $M^i$, since the existence of an $r$ with the above properties is a true fact, as witnessed by $q$, which can be expressed by a sentence with parameters in $M^i$. And given $r\in M^i$ as above, the amalgamation $q'$ of $r$ and $q$ described earlier is a condition in $\mtcl Q$. For this, it is enough to show the following claim, as all other clauses in the definition of $\mtcl Q$-condition are clear.

\begin{claim} $\mtcl N_q\cup\mtcl N_r$ is $\vec{\mtcl C}$-stratified and compatible with $F$. 
\end{claim}

\begin{proof}
We can show, by an argument similar to the one in the proof of Claim \ref{cl8}, that for every $N\in \mtcl N_r$ and $M\in\mtcl N_q$, if $\d_{M^i}\leq \d_M$, then $\a:=\max(\cl(N\cap\o_2)\cap\cl(M\cap\o_2))$ exists, $\a\in M$, and $N\cap\a=\bigcup_{\g\in C_{\a, \n}}e_\g``\d_N$ for some $\n<\d_M$.  And, by construction, $\{N\in\mtcl N_q\,:\,\d_N\in\Delta\}\cup\mtcl N_r$ is $\vec{\mtcl C}$-stratified and compatible with $F$. Putting these two facts together we get the conclusion. 
\end{proof}

 This finishes the proof of the lemma since $q'$ extends both $q$ and $r$.  
\end{proof}

\begin{remark}
Unlike in the proof of Lemma \ref{aleph15}, our models $M^0$ and $M^1$ in the above proof need not be of the form $M^i=\bigcup_{i<\d_{M^i}}M^i_\nu$, with $(M^i_\nu)_{\nu<\d_{M^i}}$ being a continuous $\in$-chain of elementary submodels containing the relevant objects.  This is of course thanks to the presence of clause (3) in the definition of $\vec{\mtcl C}$-stratified family.
\end{remark}

The following is an immediate corollary from Lemma \ref{Q-proper}. 

\begin{corollary}\label{Q-proper-cor} 
$\mtcl Q$ is proper with respect to $\mtcl K^{\vec e}_{\vec{\mtcl C}}$. 
\end{corollary}


\begin{lemma}\label{Q-cc} $\mtcl Q$ has the $\al_2$-c.c. \end{lemma}

\begin{proof}
Suppose, towards a contradiction, that $(q_i\,;\,i<\l)$, for some cardinal $\l\geq\o_2$, is a one-to-one enumeration of a maximal antichain $A$ of $\mtcl Q$. Let $\t$ be a large enough cardinal. For every $i<\o_2$ let $M_i$ be a countable elementary submodel of $H(\t)$ belonging to $\mtcl K^{\vec e}_{\vec{\mtcl C}}$ and such that $p_i$, $F$, $\vec{\mtcl C}$, $\vec e$, $A\in M_i$. 

Let $P$ be an elementary submodel of some higher $H(\chi)$ such that $|P|=\al_1$ and $\vec{\mtcl C}$, $((q_i, M_i)\,:\,i<\l)\in P$. Since all $q_i$ are distinct and $\l\geq\o_2$, we may find $i_0$ such that $q_{i_0}\notin  P$. Now, working in $P$ and since $M_{i_0}\cap P\in P$ as $M_{i_0}\in\mtcl K^{\vec e}_{\vec{\mtcl C}}$, we may find $i_1\in P\cap\l$ such that $\d_{M_{i_0}}=\d_{M_{i_1}}$ and $\{M_{i_0}, M_{i_1}\}$ is $\vec{\mtcl C}$-stratified. By Lemma \ref{Q-proper} there is a condition $q^*\in\mtcl Q$ extending $q_{i_0}$ and such that $q^*$ is $(M_{i_1}, \mtcl Q)$-generic. But now, since $A\in M_{i_1}$ is a maximal antichain of $\mtcl Q$, we can find a common extension $q'$ of $q^*$ and some $q_{i_2}\in A\cap M_{i_1}$, which is a contradiction since $q_{i_2}\neq q_{i_0}$ yet $q'$ extends both $q_{i_0}$ and $q_{i_2}$.
\end{proof}

Let now $G$ be a $\mtcl Q$-generic filter. Given any $\a\in S$, let $$D^G_\a=\{\min(I)\,:\,I\in\mtcl I^q_\a,\,q\in G,\,\a\in X_q\}$$ By Lemma \ref{dens0}, $D^G_\a$ is an unbounded subset of $\a$.

\begin{lemma}\label{dens1} 
For every $\a\in S$, $D^G_\a$ is closed in $\a$. \end{lemma}

\begin{proof}
Let $\d<\a$ be a limit ordinal forced by some $q\in \mtcl Q$ with $\a\in X_q$ to be a limit point of $D^{\dot G}_\a$ and suppose, towards a contradiction, that $\d\neq \min(I)$ for any $I\in\mtcl I^q_\a$. By the choice of $q$ we may assume that there is some $I\in\mtcl I^q_\a$ such that $\min(I)<\d$. Letting $I_0$ be the unique such $I$ with $\min(I)$ maximal within $\mtcl I^q_\a$ we may now extend $q$ to a condition $q'$ such that $[\min(I_0), \d+1)\in\mtcl I^{q'}_\a$. But $q'$ forces that $D^{\dot G}_\a\cap\d\sub \min(I_0)+1<\d$, which contradicts the assumption that $q$ forced $D^{\dot G}_\a$ to be cofinal in $\d$. 
\end{proof}

It follows from Lemmas \ref{dens0}, \ref{dens2}, and \ref{dens1} together that if we aim to define $H:\o_2\into 2$ by letting $H(\eta)=F(\alpha)$ for any $\a\in S$ such that $\eta\in D^G_\a$ (and, say, $H(\eta)=0$ if there is no $\a$ as above), then $H$ is a well-defined function which uniformizes $F$ mod.\ clubs, as witnessed by $D^G_\a$ for $\a\in S$. This concludes the proof of Theorem \ref{mainthm}.

\begin{remark}
We point out that the inconsistency proof of the forcing axiom $\MM_{\al_2}(\al_2\mbox{-c.c.})$ we have given shows the impossibility of having $\FA_{\al_2}(\Gamma)$ for the class $\Gamma$ of posets which have the $\al_{1.5}$-c.c.\ with respect to families of models which are simultaneously $\vec{\mtcl C}$-stratified, for a fixed $\Box_{\o_1, \o_1}$-sequence $\vec{\mtcl C}$, and $F$-compatible for arbitrarily fixed choices of $F$. On the other hand, the methods of \cite{genMA} allow us to build models of the forcing axiom $\FA_{\al_2}(\Gamma_0)$, where $\Gamma_0$ is the class of partial orders with the $\al_{1.5}$-c.c.\ with respect to families which are $\vec{\mtcl C}$-stratified and $F$-compatible, for a fixed $\Box_{\o_1, \o_1}$-sequence $\vec{\mtcl C}$ and a fixed $F:S^2_1\into 2$. 
\end{remark}

\section{$\MA^{1.5}_{\al_2}(\mbox{stratified})$ implies $\lnot\wCC$.}\label{s3}

The goal of this section is to prove the following theorem.

\begin{theorem}\label{wcc}
$\MA^{1.5}_{\al_2}(\mbox{stratified})$ implies $\lnot\wCC$. 
\end{theorem}

We fix a sequence $\vec e=(e_\a\,:\,0<\a<\o_2)$, 
where $e_\a:\o_1\into \a$ is a surjection for each $\a$. 

We consider the following forcing notion $\mtcl R$. A condition in $\mtcl R$ is a triple $p=(f^p, (h^p_\a\,:\, \a\in X_p), \mtcl N_p)$, where:

\begin{enumerate}

\item $f^p\sub\o_1\times\o_1$ is a finite function.
\item $X_p\in [\o_2\setminus\{0\}]^{{<}\o}$
\item For each $\a\in X_p$, 
\begin{enumerate}
\item $h^p_\a\sub\o_1\times\o_1$ is a finite function which can be extended to a continuous strictly increasing function $h:\o_1\into\o_1$, and
\item for each $\n\in \dom(h^p_\a)$ we have that $h^p_\a(\n)\in \dom(f^p)$ and $\ot(e_\a``h^p_\a(\n))<f^p(h^p_\a(\n))$.
\end{enumerate}
\item $\mtcl N_p$ is a finite stratified family of countable elementary submodels of $(H(\o_2); \in, \vec e)$.
\item The following holds for each $N\in \mtcl N_p$.
\begin{enumerate}
\item $f^p\restr \d_N\sub N$;
\item $\d_N\in\dom(f^p)$ and $f^p(\d_N)\geq\ot(N\cap\o_2)$;
\item for every $\a\in X_p\cap N$, 
\begin{enumerate}
\item $h^p_\a\restr\d_N\sub N$,
\item $\d_N\in \dom(h^p_\a)$, and
\item $h^p_\a(\d_N)=\d_N$.
\end{enumerate}
\end{enumerate}
\end{enumerate}

Given $\mtcl R$-conditions $p_0$ and $p_1$, $p_1$ extends $p_0$ iff
\begin{enumerate}
\item $f^{p_0}\sub f^{p_1}$,
\item $X_{p_0}\sub X_{p_1}$, and
\item for every $\a\in X_{p_0}$, $h^{p_0}_\a\sub h^{p_1}_\a$.
.\end{enumerate}

The following density lemmas are easy.

\begin{lemma}\label{d0}
For every $p\in\mtcl R$ and every nonzero $\b<\o_2$ there is a $\mtcl R$-condition $p^*$ extending $p$ and such that $\b\in X_{p^*}$.
\end{lemma}

\begin{proof}
We may of course assume that $\b\notin X_p$. It then suffices to set $$p^*=(f^p, (h^p_\a\,:\, \a\in X_p)\cup\{(\b, \{(\d_N, \d_N)\,:\,N\in\mtcl N_p,\,\b\in N\})\}, \mtcl N_p)$$ To see that this is a condition in $\mtcl R$ it is enough to notice that if $N\in\mtcl N_p$ is such that $\b\in N$, then $\d_N\in \dom(f^p)$ and $$\ot(e_\b``\d_N)=\ot(N\cap \b)<\ot(N\cap\o_2)\leq f^p(\d_N)$$ 
\end{proof}

Lemmas \ref{d1}, \ref{d2} and \ref{d3} are obvious.

\begin{lemma}\label{d1} For every $p\in\mtcl R$, $\a\in X_p$ and $\n<\o_1$ there is some $p^*\in\mtcl R$ extending $p$ and such that $\n\in \dom(h^{p^*}_\a)$.
\end{lemma}

\begin{lemma}\label{d2} For every $p\in\mtcl R$, $\a\in X_p$, every nonzero limit ordinal $\d\in \dom(h^p_\a)$, and every $\eta<h^p_\a(\d)$ there is a condition $p^*\in\mtcl R$ extending $p$ together with some $\n\in \dom(h^{p^*}_\a)\cap\d$ such that $h^{p^*}_\a(\nu)>\eta$.   
\end{lemma}

\begin{lemma}\label{d3} For every $p\in\mtcl R$ and every $\n<\o_1$ there is a condition $p^*\in\mtcl R$ extending $p$ and such that $\n\in \dom(f^{p^*})$. \end{lemma}

It follows from Lemmas \ref{d0}--\ref{d3} together that if $G$ is $\mtcl R$-generic and we set $$f^G=\bigcup\{f^p\,:\, p\in G\}$$ and $$C^G_\a=\bigcup\{\range(h^p_\a)\,:\,p\in G,\,\a\in X_p\}$$ for each nonzero $\a<\o_2$, then $f^G:\o_1^V\into \o_1^V$ is a function, each $C^G_\a$ is a club of $\o_1^V$, and $\ot(e_\a``\n)<f^G(\n)$ for each $\a$ and $\n\in C^G_\a$. Hence, if we can show that $\mtcl R$ has the $\al_{1.5}$-c.c. with respect to finite stratified  families of models, an application of $\MA^{1.5}_{\al_2}(\mbox{stratified})$ to $\mtcl R$ will show that $\MA^{1.5}_{\al_2}(\mbox{stratified})$ implies $\lnot\wCC$.

\begin{lemma}\label{strat-wcc}
$\mtcl R$ has the $\al_{1.5}$-c.c.\ with respect to finite stratified families of models.
\end{lemma}

\begin{proof}
Let $\t$ be a large enough cardinal, let $\mtcl N^*$ be a finite stratified family of countable elementary submodels of $H(\t)$ containing $\vec e$ and $\vec f$, and let $p_0\in\mtcl R\cap N^*_0$, where $N^*_0$ is of minimum height within $\mtcl N^*$. We will prove that there is a condition $p^*\in\mtcl R$ stronger than $p_0$ such that $p^*$ is $(N^*,  \mtcl R)$-generic for each $N^*\in\mtcl N^*$. 

Let $\mtcl N=\{N^*\cap H(\o_2)\,:\,N^*\in\mtcl N^*\}$ and for every $\d\in\{\d_N\,:\,N\in\mtcl N\}$ let $$\m(\d)=\max\{\ot(N\cap\o_2)\,:\, N\in\mtcl N,\,\d_N=\d\}$$ Let $$p^*=(f^{p_0}\cup\{(\d_N, \m(\d_N))\,:\,N\in\mtcl N\}, (h^{p^*}_\a\,:\,\a\in X_{p_0}), \mtcl N_{p_0}\cup\mtcl N),$$ where $$h^{p^*}_\a=h^{p_0}_\a\cup\{(\d_N, \d_N)\,:\,N\in\mtcl N,\,\a\in N\}$$ for each $\a\in X_{p_0}$. It is easy to check that $p^*$ is a condition in $\mtcl R$ (for this it is enough to notice that $\m(\d_{N_0})<\d_{N_1}$ holds for all $N_0$, $N_1\in\mtcl N$ with $\d_{N_0}<\d_{N_1}$, which follows from the stratification of $\mtcl N$), and it of course extends $p_0$ by construction. Hence, it will be enough to show that $p^*$ is $(N^*, \mtcl R)$-generic for every $N^*\in\mtcl N^*$. Let $D\in N^*$ be a dense and open subset of $\mtcl R$ and let $p$ be an extension of $p^*$ in $D$. We will show that there is a condition in $D\cap N^*$ compatible with $p$. 

As in the proof of Lemma \ref{aleph15}, we may assume that $N^*=\bigcup_{\n<\d_{N^*}}N^*_\n$, where $(N^*_\n)_{\n<\d_{N^*}}$ is a $\sub$-continuous $\in$-chain of models.
By moving to a suitable $N^*_{\n_0}$ and arguing there as in the proof of Lemma \ref{aleph15} using the stratification of $\mtcl N_p$, we may find a condition $r\in D\cap N^*_{\n_0}$ such that 
\begin{enumerate}
\item for every $\a\in X_p\cap X_r$, $h^p_\a\cup h^r_\a$ can be extended to a strictly increasing and continuous function $h:\o_1\into \o_1$,
\item for every $\a\in X_r\cap X_q$ and every $N\in\mtcl N_p$ such that $\d_N<\d_{N^*}$, $\a\notin N$, and 
\item $\mtcl N_p\cup\mtcl N_r$ is stratified.
\end{enumerate}
Let now $$p'=(f^r\cup (f^p\restriction (\o_1\setminus\d_{N^*})), (h^{p'}_\a\,:\,\a\in X_p\cup X_r), \mtcl N_p\cup\mtcl N_r),$$ where
\begin{enumerate}
\item for every $\a\in X_p\cap X_r$, $h^{p'}_\a=h^p_\a\cup h^r_\a$,
\item for every $\a\in X_p\setminus X_r$, $h^{p'}_\a=h^p_\a$, and
\item for every $\a\in X_r\setminus X_p$, $h^{p'}_\a=h^r_\a\cup\{(\d_N, \d_N)\,:\, N\in\mtcl N_p,\,\a\in N\}$.
\end{enumerate}
Then $p'$ is a condition in $\mtcl R$, which finishes the proof of the lemma since $p'$ is of course stronger than both $p$ and $r$. 
\end{proof}

The above lemma concludes the proof of Theorem \ref{wcc}. 

The following corollary is an immediate by-product of the proof of Theorem \ref{wcc}.

\begin{corollary}
$\ZFC$ proves that there is a poset $\mtcl P$ such that
\begin{enumerate}
\item $\mtcl P$ is proper,
\item $\mtcl P$ has the $\al_2$-c.c., and
\item $\mtcl P$ forces $\lnot\wCC$.\end{enumerate}
\end{corollary} 

\section{Another proof of the inconsistency of $\MM_{\al_2}(\al_2\mbox{-c.c.})$}\label{s4}

In this final section of the paper we will give another proof of Theorem \ref{mainthm} in Section \ref{s2}. Our argument essentially follows an argument, due to Shelah, showing that stationary preserving forcing cannot be iterated, and in fact that there is a forcing iteration of length $\o$ each of whose iterands is forced to preserve stationary subsets of $\o_1$ but such that any limit of it collapses $\o_1$.  

Let us assume that $\MM_{\al_2}(\al_2\mbox{-c.c.})$ holds. Hence, $\MA_{1.5}(\mbox{stratified})$ holds as well and so, by Theorem \ref{wcc}, there is a function $f:\o_1\into \o_1$ such that $\{\n<\o_1\,:\,g(\n)<f(\n)\}$ contains a club for every nonzero $\a<\o_2$ and every canonical function $g$ for $\a$. We will build a sequence $(f_n)_{n<\o}$ of functions from $\o_1$ to $\o_1$, together with clubs $C_n$ of $\o_1$, such that for every $n$ and every $\n\in C_n$, $f_{n+1}(\n)<f_n(\n)$. This of course will yield a contradiction since then, if $\n\in \bigcap_n C_n$, then $f_{n+1}(\n)<f_n(\n)$ for all $n$, which is impossible. 

We will make sure that the construction can keep going by arranging, for every $n<\o$ and every nonzero $\a<\o_2$, that $f_n$ dominates every canonical function for $\a$ on a club. We start our construction by letting $f_0=f$.  

Given $n<\o$ and assuming $f_n$ has been constructed, we will find $f_{n+1}$ by an application of $\MM_{\al_2}(\al_2\mbox{-c.c.})$ to the following slight variant $\mtcl R_{f_n}$ of the poset $\mtcl R$ in the proof of Theorem \ref{wcc}.

Let $\mtcl K^{f_n}$ be the collection of countable $N\elsub (H(\o_2); \in, \vec e, f_n)$ such that $\ot(N\cap\o_2)<f_n(\d_N)$. 

A condition in $\mtcl R_{f_n}$ is a tuple $p=(f^p, d_p, (h^p_\a\,:\,\a\in X_p), \mtcl N_p)$ with the following properties.

\begin{enumerate}
\item $f^p\sub\o_1\times\o_1$ is a finite function.
\item $d_p\sub\o_1\times\o_1$ is a finite function which can be extended to a continuous strictly increasing function $d:\o_1\into\o_1$.
\item For every $\n\in\dom(d_p)$, $f^p(d_p(\n))<f_n(d_p(\n))$.
\item $X_p\in [\o_2\setminus\{0\}]^{{<}\o}$
\item For each $\a\in X_p$, 
\begin{enumerate}
\item $h^p_\a\sub\o_1\times\o_1$ is a finite function which can be extended to a continuous strictly increasing function $h:\o_1\into\o_1$, and
\item for each $\n\in \dom(h^p_\a)$ we have that $h^p_\a(\n)\in \dom(f^p)$ and $\ot(e_\a``h^p_\a(\n))<f^p(h^p_\a(\n))$.
\end{enumerate}
\item $\mtcl N_p$ is a finite stratified family of members of $\mtcl K^{f_n}$.
\item The following holds for each $N\in \mtcl N_p$.
\begin{enumerate}
\item $f^p\restr \d_N\sub N$;
\item $\d_N\in\dom(f^p)$ and $f^p(\d_N)\geq\ot(N\cap\o_2)$;
\item $d_p\restr\d_N\sub N$, $\d_N\in \dom(d_p)$, and $d_p(\d_N)=\d_N$.
\item for every $\a\in X_p\cap N$, 
\begin{enumerate}
\item $h^p_\a\restr\d_N\sub N$,
\item $\d_N\in \dom(h^p_\a)$, and
\item $h^p_\a(\d_N)=\d_N$.
\end{enumerate}
\end{enumerate}
\end{enumerate}

Given $\mtcl R_{f_n}$-conditions $p_0$ and $p_1$, $p_1$ extends $p_0$ iff
\begin{enumerate}
\item $f^{p_0}\sub f^{p_1}$,
\item $d_{p_0}\sub d_{p_1}$, 
\item $X_{p_0}\sub X_{p_1}$, and
\item for every $\a\in X_{p_0}$, $h^{p_0}_\a\sub h^{p_1}_\a$.
\end{enumerate}

We now have the following density lemmas. Lemma \ref{d0'} is proved by the same argument as in the proof of Lemma \ref{d0} using the fact that all models of $\mtcl N_p$ are in $\mtcl K^{f_n}$, and Lemmas \ref{d1'}--\ref{d4'} are straightforward.

\begin{lemma}\label{d0'}
For every $p\in\mtcl R_{f_n}$ and every nonzero $\b<\o_2$ there is a $\mtcl R_{f_n}$-condition $p^*$ extending $p$ and such that $\b\in X_{p^*}$.
\end{lemma}

\begin{lemma}\label{d1'} For every $p\in\mtcl R_{f_n}$, $\a\in X_p$ and $\n<\o_1$ there is some $p^*\in\mtcl R_{f_n}$ extending $p$ and such that $\n\in \dom(h^{p^*}_\a)$.
\end{lemma}

\begin{lemma}\label{d2'} For every $p\in\mtcl R_{f_n}$, $\a\in X_p$, every nonzero limit ordinal $\d\in \dom(h^p_\a)$, and every $\eta<h^p_\a(\d)$ there is a condition $p^*\in\mtcl R_{f_n}$ extending $p$ together with some $\n\in \dom(h^{p^*}_\a)\cap\d$ such that $h^{p^*}_\a(\nu)>\eta$.   
\end{lemma}

\begin{lemma}\label{d3'} For every $p\in\mtcl R_{f_n}$ and every $\n<\o_1$ there is a condition $p^*\in\mtcl R_{f_n}$ extending $p$ and such that $\n\in \dom(f^{p^*})$. \end{lemma}

\begin{lemma}\label{d4'} For every $p\in\mtcl R_{f_n}$ and every $\n<\o_1$ there is a condition $p^*\in\mtcl R_{f_n}$ extending $p$ and such that $\n\in \dom(d_{p^*})$. \end{lemma}


It follows from the above density lemmas that if $G$ is $\mtcl R_{f_n}$-generic and we let $$D^G=\bigcup\{\range(d_p)\,:\, p\in G\},$$ $$f^G=\bigcup\{f^p\,:\,p\in G\},$$ and $$C^G_\a=\bigcup\{\range(h^p_\a)\,:\,p\in G,\,\a\in X_p\}$$ for each nonzero $\a<\o_2$, then $D^G$ is a club of $\o_1^V$, $f^G:\o_1^V\into \o_1^V$ is a function, each $C^G_\a$ is a club of $\o_1^V$, $\ot(e_\a``\n)<f^G(\n)$ for each $\a$ and $\n\in C^G_\a$, and $f^G(\n)<f_n(\n)$ for each $\n\in D^G$. It follows that an application of $\MM_{\al_2}(\al_2\mbox{-c.c.})$ to $\mtcl R_{f_n}$ will provide us with $f_{n+1}$. Hence, we just need to prove that $\mtcl R_{f_n}$ preserves stationary subsets of $\o_1$ and has the $\al_2$-c.c. This we will prove by means of the following version of Lemma \ref{Q-proper} in Section \ref{s2}.

\begin{lemma}\label{R-proper}
 Let $\t$ be a cardinal such that $\mtcl R_{f_n}\in H(\t)$ and let $M^0$, $M^1\prec H(\t)$ be countable models of the same height and such that $\mtcl R_{f_n}\in M^0\cap M^1$ and $M^0\cap H(\o_2)$, $M^1\cap H(\o_2)\in \mtcl K^{f_n}$. Then for every $p_0\in\mtcl R_{f_n}\cap M^0$ there is an extension $p^*\in \mtcl Q$ of $p_0$ such that $p^*$ is $(M^i, \mtcl R_{f_n})$-generic for $i=0$, $1$. 
\end{lemma}

\begin{proof}
Let $\m=\max\{\ot(M^0\cap\o_2), \ot(M^1\cap\o_2)\}$ and $$p^*=(f^{p^*},  d_{p^*}, (h^{p^*}_\a\,:\,\a\in X_{p_0}), \mtcl N_{p_0}\cup\{M^0\cap H(\o_2), M^1\cap H(\o_2)\}),$$ where 
\begin{enumerate}
\item $f^{p^*}=f^{p_0}\cup\{(\d_{M^0}, \m)\}$,
\item $d_{p^*}=d_{p_0}\cup\{(\d_{M^0}, \d_{M^0})\}$, and
\item $h^{p^*}_\a=h^{p_0}_\a\cup\{(\d_{M^0}, \d_{M^0})\}$ for each $\a\in X_{p_0}$.
\end{enumerate}

Using the fact that $M^0\cap H(\o_2)$, $M^1\cap H(\o_2)\in\mtcl K^{f_n}$, it is immediate to check that $p^*\in\mtcl R_{f_n}$, and it of course extends $p_0$. It will thus suffice to show that $p^*$ is $(M^i, \mtcl R_{f_n})$-generic for $i=0$, $1$. For this, let $D\in M^i$ be a dense and open subset of $\mtcl R_{f_n}$ and let $p\in D$ be an extension of $p^*$. We will show that there is a condition $r\in D\cap M^i$ compatible with $p$.

We can find $r$ by arguing, in $M^i$, in the same way as in the reflection argument in the proof of Lemma \ref{strat-wcc}. More specifically, and exactly as in that proof, we may assume that $M^i=\bigcup_{\n<\d_{M^i}}M^i_\n$, where $(M^i_\n)_{\n<\d_{M^i}}$ is a $\sub$-continuous $\in$-chain of models. Then, by moving to a suitable $M^i_{\n_0}$ and arguing there as in the proof of Lemma \ref{aleph15} using the stratification of $\mtcl N_p$, we may find a condition $r\in D\cap M^i_{\n_0}$ such that 
\begin{enumerate}
\item for every $\a\in X_p\cap X_r$, $h^p_\a\cup h^r_\a$ can be extended to a strictly increasing and continuous function $h:\o_1\into \o_1$,
\item for every $\a\in X_r\cap X_q$ and every $N\in\mtcl N_p$ such that $\d_N<\d_{M^i}$, $\a\notin N$, and 
\item $\mtcl N_p\cup\mtcl N_r$ is stratified.
\end{enumerate}

Let us define $h^{p'}_\a$, for $\a\in X_p\cup X_r$, as follows:

\begin{enumerate}
\item for every $\a\in X_p\cap X_r$, $h^{p'}_\a=h^p_\a\cup h^r_\a$;
\item for every $\a\in X_p\setminus X_r$, $h^{p'}_\a=h^p_\a$;
\item for every $\a\in X_r\setminus X_p$, $h^{p'}_\a=h^r_\a\cup\{(\d_N, \d_N)\,:\, N\in\mtcl N_p,\,\a\in N\}$.
\end{enumerate}

We then have that $$p'=(f^r\cup (f^p\restriction (\o_1\setminus\d_{M^i})), d_p\cup d_r, (h^{p'}_\a\,:\,\a\in X_p\cup X_r), \mtcl N_p\cup\mtcl N_r)$$ is a common extension in $\mtcl R_{f_n}$ of $p$ and $r$, which finishes the proof of the lemma. 
\end{proof}

\begin{lemma}\label{proj-stat-wcc} $\mtcl K^{f_n}$ is projective stationary. \end{lemma}

\begin{proof} 
Given a cardinal $\t\geq\o_2$, a function $F:[\o_2]^{{<}\o}\into \o_2$, and a stationary set $S\sub\o_1$, it is enough to show that there is a countable $X\sub\o_2$ closed under $F$ such that $\d:=X\cap\o_1\in S$ and such that $f_n(\d)>\ot(X)$. 

In order to find such an $X$ we first pick $\a\in S^2_0$ above $\o_1$ such that $F``[\a]^{{<}\o}\sub \a$. We then let $N$ be a countable elementary submodel of some larger $H(\chi)$ containing $F$, $\a$, $f_n$, and $\vec f$ and such that $\d_N\in S$, and let $X=N\cap\a$. Then $F``[X]^{{<}\o}\sub X$, $X\cap\o_1=\d_N\in S$, and $f_n(\d_N)\geq \ot(N\cap \o_2)>\ot(N\cap \a)=\ot(X)$, where the first inequality follows from the fact that for every nonzero $\b\in N\cap \o_2$ there is a club $C\in N$ of $\o_1$ such that $\ot(e_\b``\n)<f_n(\n)$ for all $\n\in C$, which implies that $\ot(N\cap\b)=\ot(e_\b``\d)<f_n(\d)$ as $\d\in C$. \end{proof}

We now have the following corollary from Lemmas  \ref{R-proper} and \ref{proj-stat-wcc}.

\begin{corollary}\label{cor000}
$\mtcl R_{f_n}$ is proper with respect to $\mtcl K^{f_n}$ and therefore it preserves stationary subsets of $\o_1$.\end{corollary}

We also have the following corollary from Lemma \ref{R-proper}.

\begin{lemma}\label{al2wcc}
$\mtcl R_{f_n}$ has the $\al_2$-c.c.
\end{lemma}

\begin{proof}
This is similar to the proof of Lemma \ref{Q-cc}. Suppose $(p_i\,;\,i<\l)$, for some cardinal $\l\geq\o_2$, is a one-to-one enumeration of a maximal antichain $A$ of $\mtcl R_{f_n}$. Let $\t$ be a large enough cardinal and for every $i<\o_2$ let $M_i$ be a countable elementary submodel of $H(\t)$ such that $p_i$, $A\in M_i$ and such that $M_i\cap H(\o_2)\in \mtcl K^{f_n}$. 

Let $P$ be an elementary submodel of some higher $H(\chi)$ such that $|P|=\al_1$ and $((p_i, M_i)\,:\,i<\l)\in P$. We may then find $i_0$ such that $p_{i_0}\notin  P$. Working in $P$, we may find $i_1\in P\cap\l$ such that $\d_{M_{i_0}}=\d_{M_{i_1}}$. By Lemma \ref{R-proper} there is a condition $p^*\in \mtcl R_{f_n}$ extending $p_{i_0}$ and such that $p^*$ is $(M_{i_1}, \mtcl R_{f_n})$-generic. Since $A\in M_{i_1}$ is a maximal antichain of $\mtcl R_{f_n}$, we can find a common extension $p'$ of $p^*$ and some $p_{i_2}\in A\cap M_{i_1}$, which is a contradiction since $p_{i_2}\neq p_{i_0}$ yet $p'$ extends both $p_{i_0}$ and $p_{i_2}$.
\end{proof}

Lemmas \ref{cor000} and \ref{al2wcc} complete our second proof of Theorem \ref{mainthm}.

\end{document}